\documentclass[onefignum,onetabnum]{siamart171218}

\usepackage{lipsum}
\usepackage{amsfonts}
\usepackage{graphicx}
\usepackage{epstopdf}
\usepackage{algorithm}
\usepackage{algorithmic}
\usepackage{tikz}
\usetikzlibrary{arrows}

\usepackage{multirow}
\usepackage{listings}
\usepackage{mathtools}
\usepackage{latexsym,amsmath,amsfonts,amscd}
\usepackage{subfigure}
\usepackage{bbm}
\newtheorem{rmk}{Remark}

\ifpdf
  \DeclareGraphicsExtensions{.eps,.pdf,.png,.jpg}
\else
  \DeclareGraphicsExtensions{.eps}
\fi

\headers{A MONOTONE $Q^1$ FEM FOR elliptic equations}{H. Li and X. Zhang}

\title{A monotone $Q^1$ finite element method for anisotropic elliptic equations}

\author{Hao Li \thanks{Department of Applied Mathematics, The Hong Kong Polytechnic University, Hong Kong.
 Email: \email{hao94.li@polyu.edu.hk}}.
\and Xiangxiong Zhang
\thanks{Department of Mathematics,
Purdue University,
150 N. University Street,
West Lafayette, IN 47907-2067, USA. 
Email: \email{zhan1966@purdue.edu}.}
}

\usepackage{amsopn}

\begin{document}

\maketitle

\begin{abstract}
We construct a monotone continuous $Q^1$ finite element method on the uniform mesh for the anisotropic diffusion problem with a diagonally dominant diffusion coefficient matrix. The monotonicity implies the discrete maximum principle. Convergence of the new scheme is rigorously proven. On quadrilateral meshes, the matrix coefficient conditions translate into specific mesh constraints.
\end{abstract}

\begin{keywords}
Inverse positivity, $Q^1$ finite element method, monotonicity, discrete maximum principle, anisotropic diffusion 
\end{keywords}

\begin{AMS}
65N30, 65N15,  65N12
\end{AMS}

\section{Introduction}
\subsection{Monotonicity and discrete maximum principle}
Consider solving the following elliptic equation on $\Omega=(0,1)^2$ with 
Dirichlet boundary conditions:
\begin{equation}
\label{elli-pde-1}
 \begin{split}
 \mathcal L u\equiv -\nabla\cdot(\mathbf{a}\nabla u )+cu=f & \quad \mbox{on} \quad \Omega, \\
 u=g & \quad \mbox{on}\quad  \partial\Omega,
\end{split}
\end{equation}
where the diffusion matrix $\mathbf{a}(\mathbf x)\in \mathbb{R}^{2 \times 2}$, 
$c(\mathbf{x})$, $f(\mathbf{x})$ and $g(\mathbf{x})$ are sufficiently smooth functions over $\bar{\Omega}$ or $\partial\Omega$. 
We assume that $\forall \mathbf{x} \in \Omega$, $\mathbf{a}(\mathbf{x})$ is symmetric and uniformly positive definite on $\Omega$. 
In the literature, \eqref{elli-pde-1} is called a heterogeneous anisotropic diffusion problem when the eigenvalues of $\mathbf{a}(\mathbf{x})$ are unequal and vary over on $\Omega$.
For a smooth function $u \in C^2(\Omega) \cap C(\bar{\Omega})$, a maximum principle holds \cite{evans}: $$\mathcal{L} u \leq 0 \quad\mbox{on} \quad\Omega \quad \Longrightarrow  \quad \max _{\bar{\Omega}} u \leq \max \left\{0, \max _{\partial \Omega} u\right\}.$$ In particular,
\begin{equation}\label{DMP}
\mathcal{L} u=0 \text { in } \Omega \Longrightarrow|u(\mathbf{x})| \leq \max _{\partial \Omega}|u|, \quad \forall(\mathbf{x}) \in \Omega.
\end{equation}

The anisotropic diffusion problem \eqref{elli-pde-1} arises from various areas of science and engineering, including plasma physics, Lagrangian hydrodynamics, and image processing. To avoid spurious oscillations or non-physical numerical solution, it is desired to have numerical schemes to satisfy \eqref{DMP} in the discrete sense.  We are interested in a linear approximation to $\mathcal{L}$ which can be represented as a matrix $L_h$. The matrix $L_h$ is called monotone if its inverse only has nonnegative entries, i.e., $L_h^{-1} \geq 0$.  Monotonicity of the scheme is a sufficient condition for the discrete maximum principle and has various applications espeically for parabolic problems, see \cite{bramble1969convergence, xu1999monotone, li2020monotonicity, hu2023, shen2022discrete, liu2022structure, AAM-40-161, zhang2024monotonicity,liu2023positivitypreserving, liu2022structure, doi:10.1137/18M1208551,li2023high}.

\subsection{Monotone schemes for anisotropic diffusion equations}
Monotone (or positive-type in some literature) numerical methods for problem \eqref{elli-pde-1} have received considerable attention, e.g., see \cite{kuzmin2009constrained, li2010anisotropic, li2007mesh, lipnikov2007monotone, liska2008enforcing, mlacnik2006unstructured, yuan2008monotone, sharma2007preserving, le2009nonlinear, cances2013monotone, ngo2016monotone}. The major efforts of studying linear monotone schemes take advantage of $M$-matrix (see \cite{plemmons1977m} for the definition), either by showing the coefficient matrix is $M$-matrix directly or the coefficient matrix can be factorized into a product of $M$-matrices. In the following, we call a numerical scheme satisfying \textit{$M$-matrix property} if the corresponding coefficient matrix is an $M$-matrix. 

By factorizing the stiffness matrix into a product of $M$-matrices, the monotonocity can still be ensured. For a nine-point scheme on a two-dimensional quadrilateral grid, the matrix condition for monotonicity with specific splitting strategy in \cite{nordbotten2007monotonicity} aligns with the Lorenz's condition presented in \cite{lorenz1977inversmonotonie, li2020monotonicity}. The difference is that in \cite{lorenz1977inversmonotonie, li2020monotonicity}, only the existence of the factorization was proved while in \cite{nordbotten2007monotonicity} the exact matrix factorization was found explicitly.

In \cite{motzkin1952approximation}, it is proved that a monotone finite difference scheme exists for any linear second-order elliptic problem on fine enough uniform mesh but a finite difference method with fixed stencil for all the problems satisfying the $M$-matrix property does not exist. With nonnegative directional splittings, \cite{weickert1998anisotropic, greenspan1965non, ngo2016monotone} propose to construct finite difference schemes for elliptic operators in the non-divergence form and divergence form. Particularly in \cite{ngo2016monotone}, it is shown that a monotone scheme satisfying the $M$-matrix property can be constructed for continuous diffusion matrix for sufficiently fine mesh and sufficiently large finite difference stencil.

In \cite{li2010anisotropic}, for the $P^1$ finite elements in two and three dimensions, the author generalized the well known non-obtuse angle condition for anisotropic diffusion problem in the sense to have the dihedral angles of all mesh elements, measured in a metric depending on $\mathbf{a}(\mathbf{x})$, be non-obtuse. It reduces to the non-obtuse angle condition for isotropic diffusion matrices when $\mathbf{a}(\mathbf{x}) = \alpha(\mathbf{x})\mathbb{I}$, where $\mathbb{I}$ is the identity matrix. The formulation was also utilized in \cite{li2010anisotropic} for the construction of the so called $M$-uniform meshes on which the numerical scheme is monotone. The approach to show monotonicity in \cite{li2010anisotropic} is to write the global matrix as the sum of local contributions. In \cite{huang2011discrete},  the Delaunay condition is extended to anisotropic diffusion problems through a refined analysis studying the whole stiffness matrix for the two-dimensional situation.
The analysis of \cite{li2010anisotropic} was extended to the anisotropic diffusion--convection--reaction problems in \cite{lu2014maximum}.

For the $Q^1$ finite elements, research on monotonicity has predominantly been focused on meshes whose elements are rectangular blocks. For the two-dimensional Poisson equation, it was noted in \cite{christie1984maximum} that the $M$-matrix property is violated when the aspect ratio, i.e. the ratio between the length of the longer edge and the shorter edge of the element, becomes excessively large. Then the discrete maximum principle is not guaranteed.

\subsection{Contributions and organization of the paper}
It is well known that the second-order accurate linear schemes, such as mixed finite element and multi-point flux approximation, do not always satisfy monotonicity on distorted meshes or with high anisotropy ratio. In this paper, we construct a monotone $Q^1$ finite element method  for solving the equation \eqref{elli-pde-1}, which is second-order accurate for function values. 

To analyze the monotonicity of the stiffness matrix, we approximate  integrals with a specific quadrature rule, particularly, the linear combination of the trapezoid rule and midpoint rule. Then we demonstrate that the continuous $Q^1$ finite element method with the specific quadrature rule, when applied to the anisotropic diffusion problem on a uniform mesh, ensures monotonicity for the problem with a diagonally dominant diffusion coefficient matrix. The method is linear and second-order accurate. The convergence of the function values for this method is rigorously proven. The coefficient constraints become mesh constraints when this $Q^1$ finite element method is applied on general quadrilateral meshes. 

This paper is organized as follows. In Section \ref{sec-preliminaries}, we introduce the notations and review some standard quadrature error estimates. In Section \ref{sec-Q1-scheme}, we derive the $Q^1$ scheme for anisotropic diffusion equation with Dirichlet boundary condition and the coefficient constraints for the stiffness matrix to be an $M$-matrix. In Section \ref{sec-convergence}, the convergence rate of function values is proved. In Section \ref{sec-gen-quad-mesh}, we discuss the extension to general quadrilateral meshes. Numerical results are given in Section \ref{sec-test}.

\section{Preliminaries}\label{sec-preliminaries}
\subsection{Notation and tools}
We introduce some notation and useful tools as follows. 
\begin{itemize}
\item For the problem dimension $d$, though we only consider the case $d = 2$, sometimes we keep the general notation $d$ to illustrate how the results are influenced by the dimension.
\item For the $Q^1$  finite element space, i.e., tensor product of linear polynomials,  the local space is defined on a reference element $\hat{K}$, e.g., $\hat{K}=[0,1]^2$. Then, the finite element space on a physical mesh element $e$ is given by the reference map from $\hat{K}$ to $e$.  The reference element $\hat{K}$ is as Figure \ref{ref-element}.
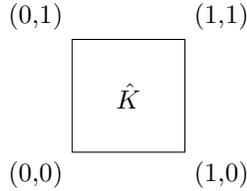
\begin{figure}[h]\label{ref-element}
\centering
\begin{tikzpicture}
\draw

 (0,1.5)  coordinate (2)
    (0,0) coordinate (1);
  \draw
    (0,0) rectangle node[] {$\hat K$} 
    (1.5,1.5)  coordinate (3)
    (1.5,0) coordinate (4)
  ;
  \path
    (1) node[below left] {(0,0)}
    (2) node[above left] {(0,1)}
    (3) node[above right] {(1,1)}
    (4) node[below right] {(1,0)}
  ;
\end{tikzpicture}
\caption{The reference element.}
\end{figure}

On  a reference element $\hat K$, we have the Lagrangian basis $\hat \phi_{0,0},\, \hat \phi_{0,1},\, \hat \phi_{1,1}, \hat \phi_{1,0}$ as 
\begin{equation}\label{basis-func}
\hat{\phi}_{0,0} = (1-\hat x_1)(1 - \hat x_2),
\quad
\hat \phi_{0,1} = (1-\hat x_1)\hat x_2,
\quad
\hat \phi_{1,1} = \hat x_1 \hat x_2,
\quad
\hat \phi_{1,0} = \hat x_1(1- \hat x_2).
\end{equation}

\item We will use ${ }^{\wedge}$ for a function to emphasize the function is defined on or transformed to the reference element $\hat{K}$ from a physical mesh element.

\item For a quadrilateral element $e$, we assume $\mathbf{F}_e= (F_{e1}, F_{e2})^T$ is the bilinear mapping such that $\mathbf{F}_e(\hat{K})=e$.
Let $\mathbf{c}_{i,j}, \, i,j = 0,1$ be the vertices of the quadrilateral element $e$. The mapping $\mathbf{F}_e$ can be written as
$$
\mathbf{F}_e=\sum_{\ell=0}^1 \sum_{m=0}^1 \mathbf{c}_{\ell, m} \hat{\phi}_{\ell, m}.
$$

\item $Q^1(\hat K)=\left\{p(\mathbf{x})=\sum_{i=0}^1 \sum_{j=1}^1 p_{i j} \hat{\phi}_{i,j}(\hat{\mathbf{x}}), \, \hat{\mathbf{x}}\in \hat{K} \right\}$ is the set of $Q^1$ polynomials on the reference element $\hat{K}$.
\item $Q^1(e)=\left\{v_h \in H^1(e): v_h \circ \mathbf{F}_e \in Q_1(\hat K)\right\}$ is the set of $Q^1$ polynomials on an element $e$.
\item $V^h=\left\{p(\mathbf{x}) \in H^1\left(\Omega_h\right):\left.p\right|_e \in Q^1(e), \, \forall e \in \Omega_h\right\}$ denotes the continuous $Q^1$ finite element space on $\Omega_h$. 
\item $V_0^h=\left\{v_h \in V^h: v_h=0 \text{ on } \partial \Omega\right\}$
\item
Let $(f,v)_e$ denote the inner product in  $L^2(e)$ and $(f,v)$ denote the inner product in $L^2(\Omega)$:
\[(f,v)_e=\int_{e} fv\,d\mathbf{x},\quad (f,v)=\int_{\Omega} fv\,d\mathbf{x}=\sum_e (f,v)_e.\]
\item Let $\langle f,v\rangle_{e,h}$ denote the approximation to $(f,v)_e$ by the mixed quadrature defined in \eqref{mixed-quad} over element $e$ with some specified quadrature parameter and $\langle f,v\rangle_h$ denotes the approximation to $(f,v)$ by 
\[\langle f,v\rangle_h=\sum_e \langle f,v\rangle_{e,h}.\]
\item Let $E(f)$ denote the quadrature error for integrating $f(\mathbf{x})$ on element $e$. Let $\hat{E}(\hat{f})$ denote the quadrature error for integrating $\hat{f}(\hat{\mathbf{x}})=f\left(\mathbf{F}_e(\hat{\mathbf{x}})\right)$ on the reference element $\hat{K}$. Then $E(f)=h^d \hat{E}(\hat{f})$ on uniform rectangular mesh with mesh size $h$.

\item The norm and semi-norms for $W^{k, p}(\Omega)$ and $1 \leq p<+\infty$, with standard modification for $p=+\infty$ :
\begin{align*}
\|u\|_{k, p, \Omega}=&\left(\sum_{i+j \leq k} \int_{\Omega}\left|\partial_{x_1}^i \partial_{x_2}^j u(x_1, x_2)\right|^p d\mathbf{x}\right)^{1 / p}, \\ 
|u|_{k, p, \Omega}=&\left(\sum_{i+j=k} \int_{\Omega}\left|\partial_{x_1}^i \partial_{x_2}^j u(x_1, x_2)\right|^p d\mathbf{x}\right)^{1 / p}, \\ 
[u]_{k, p, \Omega}=&\left(\int_{\Omega}\left|\partial_{x_1}^k u(x_1, x_2)\right|^p d\mathbf{x}+\int_{\Omega}\left|\partial_{x_2}^k u(x_1, x_2)\right|^p d\mathbf{x}\right)^{1 / p} .
\end{align*}

\item In the special case where $\omega=\Omega$, we drop the subscript, i.e. $(\cdot, \cdot):=(\cdot, \cdot)_{\Omega}$ and $\|\cdot\|:=\|\cdot\|_{\Omega}$.

\item For any $v_h \in V^h, 1 \leq p<+\infty$ and $k \geq 1$, we will abuse the notation to denote the broken Sobolev norm and semi-norms by the following symbols
\begin{align*}
\left\|v_h\right\|_{k, p, \Omega}:=&\left(\sum_e\left\|v_h\right\|_{k, p, e}^p\right)^{\frac{1}{p}},  \\
\left|v_h\right|_{k, p, \Omega}:=&\left(\sum_e\left|v_h\right|_{k, p, e}^p\right)^{\frac{1}{p}}, \\
\left[v_h\right]_{k, p, \Omega}:=&\left(\sum_e\left[v_h\right]_{k, p, e}^p\right)^{\frac{1}{p}}.
\end{align*}
\item  For simplicity, sometimes we may use $\|u\|_{k, \Omega},|u|_{k, \Omega}$ and $[u]_{k, \Omega}$ denote norm and semi-norms for $H^k(\Omega)=W^{k, 2}(\Omega)$. When there is no confusion, $\Omega$ may be dropped in the norm and semi-norms, e.g., $\|u\|_k:=\|u\|_{k, \Omega}$.

\item Inverse estimates for polynomials: there exists a constant $C_I>0$, independent of $h$ and $e$, such that for
$$
\left\|v_h\right\|_{k+1} \leq C_I h^{-1}\left\|v_h\right\|_{k}, \quad \forall v_h \in V^h, k \geq 0 .
$$

\item Elliptic regularity holds for the problem \eqref{hom-var}:
$$
\|u\|_2 \leq C\|f\|_0
$$

\item Let $\Omega_h$ is a finite element mesh for $\Omega$.  For each element $e \in \Omega_h$,  we denote $\mathbf{\bar{a}}_e = (\bar{a}^{ij}_e)$ as an approximation to the average of $\mathbf{a}$ on element $e$, i.e.  $\bar{a}^{ij}_e = \frac{1}{meas(e)}\int_e a^{ij}d\mathbf{x}$. 
Then we define piece-wise constant function $\mathbf{\bar{a}}$ on $\Omega$ as
\begin{equation*}
\mathbf{\bar{a}}(\mathbf{x}) = \mathbf{\bar{a}}_e,\quad \textrm{for $\mathbf{x} \in e$}.
\end{equation*}

\item Define the projection operator $\hat{\Pi}_1: \hat{u} \in L^1(\hat{K}) \rightarrow \hat{\Pi}_1 \hat{u} \in Q^1(\hat{K})$ by
\begin{equation}\label{q1-proj}
\int_{\hat{K}}\left(\hat{\Pi}_1 \hat{u}\right) \hat{w} d\hat{\mathbf{x}}=\int_{\hat{K}} \hat{u} \hat{w}  d\hat{\mathbf{x}}, \quad \forall \hat{w} \in Q^1(\hat{K}) .
\end{equation}

Observe that all degrees of freedom of $\hat{\Pi}_1 \hat{u}$ can be expressed as a linear combination of $\int_{\hat{K}} \hat{u} \hat{p} d \hat{\mathbf{x}}$ where $\hat{p}(\mathbf{x})$ takes the forms $1, \hat{x}_1, \hat{x}_2$, and $\hat{x}_1 \hat{x}_2$. This implies that the $H^1(\hat{K})$ (or $H^2(\hat{K})$) norm of $\hat{\Pi}_1 \hat{u}$ is dictated by $\int_{\hat{K}} \hat{u} \hat{p} d \hat{\mathbf{x}}$.
Utilizing the Cauchy-Schwartz inequality, we deduce:
$$
\left|\int_{\hat{K}} \hat{u} \hat{p} d \hat{\mathbf{x}}\right| \leq\|\hat{u}\|_{0,2, \hat{K}}\|\hat{p}\|_{0,2, \hat{K}} \leq C\|\hat{u}\|_{0,2, \hat{K}}
$$
From which it follows that:
$$
\left\|\Pi_1 \hat{u}\right\|_{1,2, \hat{K}} \leq C\|\hat{u}\|_{0,2, \hat{K}}
$$
This establishes that $\hat{\Pi}_1$ acts as a continuous linear mapping from $L^2(\hat{K})$ to $H^1(\hat{K})$. Similarly, by extending this argument, we can also demonstrate that $\hat{\Pi}_1$ is a continuous linear mapping from $L^2(\hat{K})$ to $H^2(\hat{K})$.

\item We denote all the the vertices of $\Omega_h$ inside $\Omega$ by $\mathbf{x}_j$, $j=1, \ldots, N_h$ and all the the vertices of $\Omega_h$ on $\partial\Omega$ by $\mathbf{x}_j$, $j=N_h+1, \ldots, N_h+N_h^{\partial}$. 
The corresponding Lagrange basis functions  in $V_h$ are denoted by $\varphi_i$, $i=1, \ldots N_h+N_h^{\partial}$, which are continuous in $\Omega$, linear polynomials in each element $e$ and
$$
\varphi_i\left(\mathbf{x}_j\right)=\delta_{ij},\quad i,j = 1, \ldots, N_h+N_h^{\partial}.
$$

\end{itemize}

\subsection{Mixed quadrature}
To analyze and impose the monotonicity of the stiffness matrix, we will use numerical quadrature rules to approximate integrals. As we will see, the choice of quadrature rules can significantly affect the monotonicity of the numerical schemes.

For a one-dimensional integral of function $f$ over the interval $[0, 1]$, we can approximate $\int_0^1 f(\hat{x}) d \hat{x}$ using either the trapezoid rule, given by $ \frac{f(0)+f(1)}{2}$, or the midpoint rule,
$f\left(\frac{1}{2} \right)$. Both quadrature offer second-order accuracy.  We will use the linear combination of these two kinds of quadrature as follows:
\begin{equation}\label{mixed-quad-1d} 
\begin{aligned}
\int_0^1 f(\hat{x}) d \hat{x}  \simeq &\lambda \frac{f(0)+f(1)}{2} + (1- \lambda) f\left(\frac{1}{2} \right) \\
= &\hat \omega_1f(\hat \xi_1)+\hat \omega_2 f(\hat \xi_2)  + \hat \omega_3 f(\hat \xi_1) ,
\end{aligned}
\end{equation}
where $\lambda$ is a parameter to be determined and 
\begin{equation}\label{mixed-weights-pts}
\hat \omega_1 = \frac{\lambda}{2}, \quad \hat \omega_2 = 1-\lambda, \quad \hat \omega_3 = \frac{\lambda}{2}, \quad \hat \xi_1 = 0, \quad \hat \xi_2 = \frac12, \quad \hat \xi_3 = 1.
\end{equation}
When $\lambda = 1$, the mixed quadrature recovers the trapezoid rule and when $\lambda = 0$ the mixed quadrature recovers the midpoint rule.

To approximate integration on square $\hat K$, we may use the mixed quadrature \eqref{mixed-quad-1d} with different parameters $\lambda^1$ and $\lambda^2$ for different dimension $x_1$ and $x_2$ respectively.  By Fubini's theorem, 
\begin{equation*}
\begin{aligned}
& \int_{\hat{K}} f(\hat{\mathbf{x}}) d \hat{\mathbf{x}}=\int_0^1 \int_0^1 f(\hat x_1, \hat x_2) d \hat x_1 d\hat x_2=\int_0^1\left(\int_0^1 f\left(\hat x_1, \hat x_2\right) d \hat x_2\right) d \hat x_1\\
 \simeq & \int_0^1\left(\sum_{q=1}^3 \hat{\omega}^2_q f\left(\hat x_1, \hat{\xi}_q\right)\right) d \hat x_1 
 \simeq \sum_{p=1}^{3} \hat{\omega}^1_p\left(\sum_{q=1}^{3} \hat{\omega}^2_q f\left(\hat{\xi}_p, \hat{\xi}_q\right)\right) = \sum_{p=1}^{3} \sum_{q=1}^{3} \hat{\omega}^1_p \hat{\omega}^2_q f\left(\hat{\xi}_p, \hat{\xi}_q\right),
\end{aligned}
\end{equation*}
where $\omega^j_i$ are just  $\omega_i$ but replacing $\lambda$ with $\lambda^j$ in \eqref{mixed-weights-pts} for  $i=1,2,3$, $j =1,2$.

On the reference element $\hat K$, for convenience, to denote the above quadrature for integral approximation with parameter $\boldsymbol{\lambda} = \left(\lambda^1, \lambda^2 \right)$, we will use the following notation
\begin{equation}\label{mixed-quad-ref}
\int_{\hat{K}} \hat{f}(\hat{\mathbf{x}}) d^h_{\boldsymbol{\lambda}} \hat{\mathbf{x}} :=  \sum_{p=1}^{3} \sum_{q=1}^{3} \hat{\omega}^1_p \hat{\omega}^2_q f\left(\hat{\xi}_p, \hat{\xi}_q\right).
\end{equation}

 Given the quadrature parameter $\boldsymbol{\lambda}_e = \left(\lambda_e^1, \lambda_e^2 \right)$, the quadrature approximation to $\int_{e} f(\mathbf{x}) d\mathbf{x}$ is denoted as
\begin{equation}\label{mixed-quad}
\int_{e} f(\mathbf{x}) d^h_{\boldsymbol{\lambda}_e} \mathbf{x} :=  \int_{\hat{K}} f\circ \mathbf{F}_e(\hat{\mathbf{x}}) d^h_{\boldsymbol{\lambda}_e}\hat{\mathbf{x}}.
\end{equation}

Then we define the quadrature approximation over the entire domain $\Omega$ as
\begin{equation}\label{mixed-quad-domain}
\int_{\Omega} fd^h_{\boldsymbol{\lambda}_{\Omega}}\mathbf{x} :=  \sum_{e\in \Omega_h}\int_{e} fd^h_{\boldsymbol{\lambda}_e} \mathbf{x},
\end{equation}
where $\boldsymbol{\lambda}_{\Omega} = \left( \boldsymbol{\lambda}_e\right)_{e\in \Omega_h}$ can be viewed as a vector-valued piece-wise constant function, with values $\boldsymbol{\lambda}_e$ which differ across different elements.

As a particular instance, $\int_{\Omega} fd^h_{1}\mathbf{x}$ denotes the case $\boldsymbol{\lambda}_e = (1, 0)$ for all $e\in \Omega_h$, i.e. the integral on each element are approximated by the trapezoid rule in all directions.

\subsection{Quadrature error estimates}

The Bramble-Hilbert Lemma for $Q^k$ polynomials can be stated as follows, see Exercise 3.1.1 and Theorem 4.1.3 in \cite{ciarlet2002finite}:
\begin{theorem}
\label{bh-lemma}
If a continuous linear mapping  $\hat\Pi: H^{k+1}(\hat K)\rightarrow H^{k+1}(\hat K)$
satisfies $\hat\Pi \hat v=\hat v$ for any $\hat v\in Q^k(\hat K)$, then 
\begin{equation}
\|\hat u-\hat \Pi\hat u\|_{k+1,\hat K}\leq C [\hat u]_{k+1, \hat K}, \quad \forall\hat u\in H^{k+1}(\hat K).
\label{bh1}
\end{equation}
Therefore if $l(\cdot)$ is a continuous linear form on the space $H^{k+1}(\hat K)$ satisfying
$l(\hat v)=0, \,\forall \hat v\in Q^k(\hat K),$
then \[|l(\hat u)|\leq C \|l\|'_{k+1, \hat K} [\hat u]_{k+1,\hat K},\quad \forall\hat u\in H^{k+1}(\hat K),\]
where $ \|l\|'_{k+1, \hat K}$ is the norm in the dual space of $H^{k+1}(\hat K)$.
\end{theorem}

By applying Bramble-Hilbert Lemma, we have the following quadrature estimates. 
\begin{lemma}\label{quaderror-theorem}
For  a sufficiently smooth function $a \in H^{2}(e)$,  we have 
\begin{align}
\int_e a d \mathbf{x}-\int_e ad^h \mathbf{x} =&  \mathcal{O}\left(h^{2+\frac{d}{2}}\right)[a]_{2, e}=\mathcal{O}\left(h^{2+d}\right)[a]_{2, \infty, e}\\
\int_e a d \mathbf{x}-\int_e \bar a_e d \mathbf{x} =& \mathcal{O}\left(h^{2+\frac{d}{2}}\right)[a]_{2, e}=\mathcal{O}\left(h^{2+d}\right)[a]_{2, \infty, e}
\end{align}
\end{lemma}

\begin{proof}
For any $\hat{f} \in H^2(\hat{K})$, since quadrature are represented by point values, with the Sobolev's embedding we have
$$
|\hat{E}(\hat{f})| \leq C|\hat{f}|_{0, \infty, \hat{K}} \leq C\|\hat{f}\|_{2, 2, \hat{K}}
$$
Therefore $\hat{E}(\cdot)$ is a continuous linear form on $H^2(\hat{K})$ and $\hat{E}(\hat{f})=0$ if $\hat{f} \in Q^1(\hat{K})$.  Then the Bramble-Hilbert lemma implies
$$
|E(a)|=h^d|\hat{E}(\hat{a})| \leq C h^d[\hat{a}]_{2, 2, \hat{K}}=\mathcal{O}\left(h^{2+\frac{d}{2}}\right)[a]_{2, 2, e}=\mathcal{O}\left(h^{2+d}\right)[a]_{2, \infty, e}
$$
\end{proof}

\begin{lemma}\label{rhs-quad-error}
If $f \in H^{2}(\Omega)$, $\forall v_h \in V^h$, we have
$$
\left(f, v_h\right)-\left\langle f, v_h\right\rangle_h=\mathcal{O}\left(h^{2}\right)\|f\|_{2}\left\|v_h\right\|_1.
$$
\end{lemma}
\begin{proof}
Applying Theorem \ref{bh-lemma}, on element $e$, with $\frac{\partial^2 \hat{v}_h}{\partial^2 \hat{x}_i}$ vanish, we obtain:
\begin{equation*}
\begin{aligned}
& E(fv) = h^{d}\hat{E}(\hat f \hat v_h)
\leq C h^{d} [\hat f \hat v_h]_{2, 2,\hat{K}} \\
\leq & C h^{d}\left( |\hat f|_{2, 2, \hat{K}} |\hat v_h|_{0, \infty, \hat{K}} + |\hat f|_{1, 2, \hat{K}} |\hat v_h|_{1, \infty, \hat{K}}\right) \\
\leq & C h^{d}\left( |\hat f|_{2, 2, \hat{K}} |\hat v_h|_{0, 2, \hat{K}} + |\hat f|_{1, 2, \hat{K}} |\hat v_h|_{1, 2, \hat{K}}\right) \\
\leq & C h^{2}\left( |f|_{2, 2, e} | v_h|_{0, 2, e} + |f|_{1, 2, e} |v_h|_{1, 2, e} \right)
= \mathcal{O}\left(h^{2}\right)\|f\|_{2,e}\left\|v_h\right\|_{1,e}.
\end{aligned}
\end{equation*}
By sum the above result over all elements of $\Omega_h$, then we conclude with
$$
\left(f, v_h\right)-\left\langle f, v_h\right\rangle_h=\mathcal{O}\left(h^{2}\right)\|f\|_{2}\left\|v_h\right\|_1.
$$
\end{proof}

\begin{lemma}\label{2nd-order-bh-cor}
If $u \in H^3(e)$,  for $i,j=1,2$, then $\forall v_h$, $$\int_e u_{x_i} (v_h)_{x_j}d\mathbf{x}-\int u_{x_i} (v_h)_{x_j}d_{\boldsymbol{\lambda}_e}^h\mathbf{x}=\mathcal{O}\left(h^{2}\right)\|u\|_{3,e}\left\|v_h\right\|_{2,e}.$$
\end{lemma}
\begin{proof}
Applying Theorem \ref{bh-lemma}, we obtain:
\begin{equation*}
\begin{aligned}
& E(u_{x_i} (v_h)_{x_j}) = h^{d-2}\hat{E}(\hat u_{\hat x_i} (\hat v_h)_{\hat x_j})
\leq C h^{d-2} [\hat u_{\hat x_i} (\hat v_h)_{\hat x_j}]_{2, 2,\hat{K}} \\
\leq & C h^{d-2}\left( |\hat u_{\hat x_i}|_{2, 2, \hat{K}} |(\hat v_h)_{\hat x_j}|_{0, \infty, \hat{K}} + |\hat u_{\hat x_i}|_{1, 2, \hat{K}} |(\hat v_h)_{\hat x_j}|_{1, \infty, \hat{K}} + |\hat u_{\hat x_i}|_{0, 2, \hat{K}} |(\hat v_h)_{\hat x_j}|_{2, \infty, \hat{K}}\right) \\
\leq & C h^{d-2}\left( |\hat u_{\hat x_i}|_{2, 2, \hat{K}} |(\hat v_h)_{\hat x_j}|_{0, 2, \hat{K}} + |\hat u_{\hat x_i}|_{1, 2, \hat{K}} |(\hat v_h)_{\hat x_j}|_{1, 2, \hat{K}} + |\hat u_{\hat x_i}|_{0, 2, \hat{K}} |(\hat v_h)_{\hat x_j}|_{2, 2, \hat{K}}\right)\\
\leq & C h^{d-2}\left(|\hat u|_{3, 2, \hat{K}} |\hat v_h|_{1, 2, \hat{K}} + |\hat u|_{2, 2, \hat{K}} |\hat v_h|_{2, 2, \hat{K}} \right).
\end{aligned}
\end{equation*}
where the second last inequality is implied by the equivalence of norms over $Q^1(\hat{K})$ and in the last inequality we use the fact that the third derivative of $Q^1$ polynomial vanish.

Therefore,
\begin{equation*}
\begin{aligned}
E(u_{x_i} (v_h)_{x_j}) \leq C h^{2}\left(| u|_{3, 2, e} |v_h|_{1, 2, e} + |u|_{2, 2, e} |v_h|_{2, 2, e} \right) = \mathcal{O}\left(h^{2}\right)\|u\|_{3,e}\left\|v_h\right\|_{2,e}.
\end{aligned}
\end{equation*}

\end{proof}

\begin{lemma}\label{1st-order-rhs-quad-error}
If $f \in H^2(\Omega)$ or $f \in V^h$, $\forall v_h$, we have
$$\left(f, v_h\right)-\left\langle f, v_h\right\rangle_h=\mathcal{O}\left(h\right)\|f\|_2\left\|v_h\right\|_0.$$
\end{lemma}
\begin{proof}
As in the proof of Lemma \ref{rhs-quad-error}, we have
\begin{equation*}
\begin{aligned}
E(fv) = \mathcal{O}\left(h^{2}\right)\|f\|_{2,e}\left\|v_h\right\|_{1,e}.
\end{aligned}
\end{equation*}
By applying the inverse estimate to polynomial $v_h$, we have
$$
E(fv) = \mathcal{O}\left(h\right)\|f\|_{2,e}\left\|v_h\right\|_{0,e}.
$$
Summing the previous result across all elements in $\Omega_h$, we conclude:
$$
\left(f, v_h\right)-\left\langle f, v_h\right\rangle_h=\mathcal{O}\left(h\right)\|f\|_{2}\left\|v_h\right\|_0.
$$
\end{proof}

\section{The $Q^1$ finite element method and its monotonicity}\label{sec-Q1-scheme}
In this section, we first derive the $Q^1$ finite element scheme then pursue its monotonicity.
\subsection{Derivation of the scheme}
For problem \eqref{elli-pde-1}, assuming there is a function $\bar{g} \in H^1(\Omega)$ as an extension of $g$ so that $\left.\bar{g}\right|_{\partial \Omega}=g$, the  variational form  of \eqref{elli-pde-1} is to find $\tilde{u}=u-\bar{g}\in H_0^1(\Omega)$ satisfying
\begin{equation}\label{hom-var}
 \mathcal A(u, v)=(f,v)- \mathcal A(\bar{g}, v) ,\quad \forall v\in H_0^1(\Omega),
 \end{equation}
 where $\mathcal A(u,v)=\int_{\Omega} \mathbf{a} \nabla u \cdot \nabla v d\mathbf{x}+\int_{\Omega} c u v d\mathbf{x}$, $ (f,v)=\int_{\Omega}fv d\mathbf{x}$.
 
 Let $V_0^h \subseteq H_0^1(\Omega)$ be the continuous finite element space consisting of piece-wise $Q^1$ polynomials. To have a second-order monotone method, we first approximate the matrix coefficients $\mathbf{a}=(a^{ij}(\mathbf{x}))$ by either its average $\frac{1}{meas(e)}\int_e \mathbf{a} d\mathbf{x}$ or its middle point value on each element $e$. The approximation is denoted by $\mathbf{\bar{a}}_e$. Then we obtain the modified bilinear form $$\mathcal{\bar{A}}(u, v)=\int_{\Omega} \mathbf{\bar{a}} \nabla u \cdot \nabla v d\mathbf{x}+\int_{\Omega} c u v d\mathbf{x},$$
where $\mathbf{\bar{a}}=\left(\mathbf{\bar{a}}_e\right)_{e \in \Omega_h}$. In practice, we take $\mathbf{\bar{a}}_e$ to be the middle point value of $\mathbf{a}$ on element $e$ for smooth enough $\mathbf{a}$ and fine enough mesh $\Omega_h$.

By approximating integrals in $\mathcal{\bar{A}}( u, v)$ with quadrature specified in \eqref{mixed-quad-domain}, along with designated quadrature parameter $\boldsymbol{\lambda}_{\Omega}$, we derive the following numerical scheme:  find $ u_h \in V_0^h$ satisfying
\begin{equation}\label{hom-var-num-quad}
\mathcal A_h( u_h, v_h)=\langle f,v_h \rangle_h - A_h( g_I, v_h),\quad \forall v_h\in V_0^h,
\end{equation}
where the approximated bilinear form is defined as 
\begin{equation}\label{bilinear-form-appr}
\mathcal A_h( u_h, v_h): =\int_{\Omega} \mathbf{\bar{a}} \nabla u_h \cdot \nabla v_h d^h_{\boldsymbol{\lambda}}\mathbf{x} + \int_{\Omega} c u_h  v_h d^h_{1}\mathbf{x}.
\end{equation}
The right hand side is defined as 
\begin{equation}\label{rhs-appr}
\langle f,v_h \rangle_h: =\int_{\Omega} f  v_h d^h_{1}\mathbf{x},
\end{equation}
and  $g_I \in V^h$ is the piece-wise $Q^{1}$ Lagrangian interpolation polynomial of the following function:
$$
g(x, y)=\left\{\begin{array}{lll}
0, & \text { if } & (x, y) \in(0,1)^2 \\
g(x, y), & \text { if } & (x, y) \in \partial \Omega .
\end{array}\right.
$$
Then $\bar{u}_h=u_h+g_I$ is the numerical solution for the problem \eqref{elli-pde-1}.

Obviously the quadrature parameters $\boldsymbol{\lambda} = (\lambda^1, \lambda^2)$ on each element are to be determined  for the  quadrature \eqref{mixed-quad}. It is not obvious that the numerical solution $\bar{u}_h$ is an accurate approximation of the exact solution $u$ as $\mathbf{\bar{a}}$ varies depending on the mesh.

Let us denote $\mathbf{f}$ the vector consisting of $f_{i} = f(\mathbf{x}_i)$ for $i =1,\dots,N_h$ and $\bar{\mathbf{f}}$ an abstract vector consisting of $f_{i}$ for $i =1,\dots,N_h$ and the boundary condition $g_{i} = g(\mathbf{x}_i)$ at the boundary grid points $i =N_h+1,\dots,N_h+N_h^{\partial}$. Besides, we denote 
$\bar{\mathbf{u}} = (u_{1}, \dots, u_{N_h+N_h^{\partial}})$ the vector such that $$\bar{u}_h = \sum_{i=1}^{N_h+N_h^{\partial}}u_{i}\varphi_i.$$
Then scheme \eqref{hom-var-num-quad} can be written as a finite difference scheme \cite{li2020superconvergence}, with the matrix vector representation $\bar{A} \bar{\mathbf{u}}= M\mathbf{f}$ where $\bar{A} = (a_{ij})_{N_h\times (N_h+N_h^{\partial})}$, $ a_{ij}= \mathcal A_h( \varphi_j, \varphi_i)$, $i=1,\dots, N_h$, $j=1,\dots, N_h+N_h^{\partial}$, and $M$ is the lumped mass matrix. For convenience, after inverting the lumped mass matrix $M$, with the boundary conditions, the whole scheme can be represented in a matrix vector form 
\begin{equation}\label{q1-fem-fd-form}
\bar{L}_h \bar{\mathbf{u}}=\bar{\mathbf{f}},
\end{equation}
where 
\begin{align*}
\left(\bar{L}_h \bar{\mathbf{u}}\right)_{i}:= & \left(M^{-1}\bar{A} \bar{\mathbf{u}}\right)_{i} = f_i, \quad     i =1,\dots,N_h, \\
\left(\bar{L}_h \bar{\mathbf{u}}\right)_{i}:= & u_i= g_i, \quad     i =N_h+1,\dots,N_h+N_h^{\partial}.
\end{align*}

\subsection{Discrete maximum principle}
In this subsection, we review how the monotonicity implies the discrete maximum principle. For the matrix form \eqref{q1-fem-fd-form} of the scheme \eqref{hom-var-num-quad}, with
$$
\mathbf{u}=\left(u_1, \dots , u_{N_h}\right)^T, \quad
\mathbf{u}^{\partial}=\left(
u_{N_h+1}, \dots, u_{N_h+N_h^{\partial}}
\right)^T, \quad
\bar{\mathbf{u}}=\left(
u_1, \dots, u_{N_h+N_h^{\partial}}\right)^T,
$$
we have the finite difference operator $\mathcal{L}_h$ defined by $\bar{L}_h$
$$
\mathcal{L}_h (\bar{\mathbf{u}}):= \bar{L}_h \bar{\mathbf{u}}=\bar{\mathbf{f}}, \quad \bar{L}_h=\left(\begin{array}{cc}
L_h & B^{\partial} \\
0 & I
\end{array}\right), \, \bar{\mathbf{u}}=\binom{\mathbf{u}}{\mathbf{u}^{\partial}},\,  \bar{\mathbf{f}}=\binom{\mathbf{f}}{\mathbf{g}} .
$$
The discrete maximum principle is
\begin{equation}\label{dmp}
\bar{L}_h (\bar{\mathbf{u}})_i \leq 0, \, 1 \leq i \leq N_h \Longrightarrow \max _i u_i \leq \max \left\{0, \max_{N_h+1\leq i\leq N_h+N_h^{\partial}} u_i\right\}.
\end{equation}

The following result was proven in \cite{ciarlet1970discrete}:
\begin{theorem}\label{them-dmp}
A finite difference operator $\mathcal{L}_h$ satisfies the discrete maximum principle \eqref{dmp} if $\bar{L}_h^{-1} \geq 0$ and all row sums of $\bar{L}_h$ are non-negative.
\end{theorem}

\subsection{Monotonicity of the $Q^1$ finite element}
To have the monotonicity, we are interested in conditions for $\bar{L}_h$ being an $M$-matrix. Recall a sufficient condition for $M$-matrix, see condition $C_{10}$ in \cite{plemmons1977m}:
\begin{lemma}
For a real irreducible square matrix $A$ with positive diagonal entries and non-positive off-diagonal entries, $A$ is a nonsingular $M$-matrix if all the row sums of A are non-negative and at least one row sum is positive.
\end{lemma}

Then we have the following result on the uniform rectangular mesh. The stiffness matrix of \eqref{hom-var-num-quad} is denoted as $A = (a_{ij}) = (\mathcal A_h( \varphi_j, \varphi_i))$, $i,j=1,\dots, N_h$. 
\begin{theorem}\label{thm-stif-m-matr}
Assume $\forall e\in \Omega_h$, $|\bar{a}^{12}_e| \leq \min\{\bar{a}^{11}_e,\bar{a}^{22}_e\}$. Then for the $Q^1$ scheme given by \eqref{hom-var-num-quad} for the elliptic equation \eqref{elli-pde-1} on uniform rectangular mesh, the stiffness matrix and $\bar{L}_h$ are $M$-matrices and the finite difference operator defined by $\bar{L}_h$ satisfies discrete maximum principle, 
provided the quadrature parameters for each element $e$ are chosen as:
\begin{equation}\label{param-setting}
\lambda^1_e,  \,\lambda^2_e \in \left( \frac{|\bar{a}^{11}_e -\bar{a}^{22}_e|}{\bar{a}^{11}_e+\bar{a}^{22}_e}, 1 - \frac{2|\bar{a}^{12}_e|}{\bar{a}^{11}_e+\bar{a}^{22}_e}\right].
\end{equation}
When $|\bar{a}^{12}_e| = \min\{\bar{a}^{11}_e,\bar{a}^{22}_e\}$, \eqref{param-setting} means we take $\lambda^1_e,  \lambda^2_e $ to be the upper bound of the interval, i.e. $1 - \frac{2|\bar{a}^{12}_e|}{\bar{a}^{11}_e+\bar{a}^{22}_e}$.
\end{theorem}
\begin{proof}
First, we consider the following quadrature approximation results on the reference element $\hat K$. With quadrature \eqref{mixed-quad-ref} and quadrature parameter $\boldsymbol{\lambda}_e = \left(\lambda_e^1, \lambda_e^2 \right)$, we have
\begin{align*}
\langle \mathbf{\bar{a}}\nabla \phi_{0,0},\nabla \phi_{0,1} \rangle_{h}=\langle \mathbf{\bar{a}}\nabla \phi_{1,1},\nabla \phi_{1,0}\rangle_{h} = -\frac{1}{4}(\lambda_e^2\bar{a}^{11}_e+\lambda_e^1\bar{a}^{22}_e)+\frac{1}{4}(\bar{a}^{11}_e-\bar{a}^{22}_e),\\
\langle \mathbf{\bar{a}}\nabla \phi_{0,0},\nabla \phi_{1,0} \rangle_{h}=\langle \mathbf{\bar{a}}\nabla \phi_{0,1},\nabla \phi_{1,1}\rangle_{h} = -\frac{1}{4}(\lambda_e^2\bar{a}^{11}_e+\lambda_e^1\bar{a}^{22}_e)+\frac{1}{4}(\bar{a}^{22}_e-\bar{a}^{11}_e),\\
\langle \mathbf{\bar{a}}\nabla \phi_{0,0},\nabla \phi_{1,1} \rangle_{h} = -\frac{1}{4}\left((1-\lambda_e^2)\bar{a}^{11}_e+(1-\lambda_e^1)\bar{a}^{22}_e\right) - \frac{1}{2}\bar{a}^{12}_e,\\
\langle \mathbf{\bar{a}}\nabla \phi_{0,1},\nabla \phi_{1,0} \rangle_{h} = -\frac{1}{4}\left((1-\lambda_e^2)\bar{a}^{11}_e+(1-\lambda_e^1)\bar{a}^{22}_e\right) + \frac{1}{2}\bar{a}^{12}_e.
\end{align*} 
With \eqref{param-setting} and the assumption $|\bar{a}^{12}_e| \leq \min\{\bar{a}^{11}_e,\bar{a}^{22}_e\}$, we have
\begin{equation}\label{local-stif-matr-val}
\begin{aligned}
\langle \mathbf{\bar{a}}\nabla \phi_{0,0},\nabla \phi_{0,1} \rangle_{h}=\langle \mathbf{\bar{a}}\nabla \phi_{1,1},\nabla \phi_{1,0}\rangle_{h} \in \left[\frac{1}{2}\left(|\bar{a}^{12}_e| - \bar{a}^{22}_e  \right), \frac{1}{4}(\bar{a}^{11}_e -\bar{a}^{22}_e-|\bar{a}^{11}_e -\bar{a}^{22}_e| \right ),\\
\langle \mathbf{\bar{a}}\nabla \phi_{0,0},\nabla \phi_{1,0} \rangle_{h}=\langle \mathbf{\bar{a}}\nabla \phi_{0,1},\nabla \phi_{1,1}\rangle_{h} \in \left[\frac{1}{2}\left(|\bar{a}^{12}_e| - \bar{a}^{11}_e  \right),  \frac{1}{4}(\bar{a}^{22}_e - \bar{a}^{11}_e-|\bar{a}^{11}_e -\bar{a}^{22}_e| \right),\\
\langle \mathbf{\bar{a}}\nabla \phi_{0,0},\nabla \phi_{1,1} \rangle_{h} \in \left( -\frac{1}{2}(\min\{\bar{a}^{11}_e,\bar{a}^{22}_e\}-\bar{a}^{12}_e), -\frac{1}{2}(|\bar{a}^{12}_e|+\bar{a}^{12}_e )\right],\\
\langle \mathbf{\bar{a}}\nabla \phi_{0,1},\nabla \phi_{1,0} \rangle_{h} \in \left( -\frac{1}{2}(\min\{\bar{a}^{11}_e,\bar{a}^{22}_e\}+\bar{a}^{12}_e), -\frac{1}{2}(|\bar{a}^{12}_e|-\bar{a}^{12}_e )\right],
\end{aligned}
\end{equation}
which are all non-positive. Again, when $|\bar{a}^{12}_e| = \min\{\bar{a}^{11}_e,\bar{a}^{22}_e\}$, we will take the above values as the bound of the closed side of the interval.

Given $i, j \in\left\{1, \ldots, N_h + N_h^{\partial}\right\}$, obviously, if both $\mathbf{x}_i$ and $\mathbf{x}_j$ are vertices of the same elements $e$, then we have
\begin{equation}
\begin{aligned}
& a_{ij} = \mathcal A_h( \varphi_j, \varphi_i) \\
= & \sum_{e \in \Omega_h} \int_{e} \mathbf{\bar{a}} \nabla \varphi_j \cdot \nabla \varphi_i d^h_{\boldsymbol{\lambda}_e}\mathbf{x} + \int_{e} c \varphi_j  \varphi_i  d^h_{1}\mathbf{x} \\
= & \sum_{e \in \Omega_h} \int_{\hat K} \mathbf{\bar{a}} \hat{\nabla} \hat{\varphi}_j \cdot \hat{\nabla} \hat{\varphi}_i d^h_{\boldsymbol{\lambda}_e}\hat{\mathbf{x}} + \int_{\hat K} \hat{c} \hat{\varphi}_j  \hat{\varphi}_i   d^h_{1}\hat{\mathbf{x}} \\
= & \sum_{i,j \in e} \int_{\hat K} \mathbf{\bar{a}} \hat{\nabla} \hat{\varphi}_j \cdot \hat{\nabla} \hat{\varphi}_i d^h_{\boldsymbol{\lambda}_e}\hat{\mathbf{x}} + \int_{\hat K} \hat{c} \hat{\varphi}_j  \hat{\varphi}_i d^h_{1}\hat{\mathbf{x}}
\end{aligned}
\end{equation}
where $\sum_{i,j \in e}$ means summation over all elements $e$ containing both vertices $i$ and $j$.

Notice that $\int_{\hat K} \hat{c} \hat{\varphi}_j  \hat{\varphi}_i   d^h_{1}\hat{\mathbf{x}}$ vanish if $i \neq j$ and $\int_{\hat K} \mathbf{\bar{a}} \hat{\nabla} \hat{\varphi}_j \cdot \hat{\nabla} \hat{\varphi}_i d^h_{\boldsymbol{\lambda}_e}\hat{\mathbf{x}}$ aligns with one of the values in \eqref{local-stif-matr-val} depending on their relative positions. Therefore, for $i \neq j$, with \eqref{param-setting} and the assumption $|\bar{a}^{12}_e| \leq \min\{\bar{a}^{11}_e,\bar{a}^{22}_e\}$ we have 
\begin{equation}
\begin{aligned}
A_{i j}=\sum_{i,j \in e} \int_{\hat K} \mathbf{\bar{a}} \hat{\nabla} \hat{\varphi}_j \cdot \hat{\nabla} \hat{\varphi}_i d^h_{\boldsymbol{\lambda}_e}\hat{\mathbf{x}} \leq 0.
\end{aligned}
\end{equation}


For $i=1,\dots, N_h$, we note that
\begin{equation}\label{stiff-row-sum}
\sum_{j=1}^{N_h+N_h^{\partial}} A_{i j}=\sum_{j=1}^{N_h+N_h^{\partial}}\mathcal A_h( \varphi_j, \varphi_i)= \mathcal A_h( 1,\varphi_i) = Cc_{i} \geq 0,
\end{equation}
where $C$ is a certain positive number and $c_{i} = c(\mathbf{x}_i) \geq 0$.
If $\mathbf{x}_i$ has no neighboring node on the boundary, by $\sum_{j=1}^{N_h} A_{i j}=\sum_{j=1}^{N_h+N_h^{\partial}} A_{i j}$ and \eqref{stiff-row-sum}, the $i$-th row sum of $A$ is non-negative. Therefore, we have
\begin{equation}\label{weak-diag-dominance}
A_{ii}\geq\sum_{ j=1, j \neq i}^{N_h} \left|A_{i j}\right|.
\end{equation}
If $\mathbf{x}_i$ has a neighboring node on the boundary, with \eqref{stiff-row-sum} and $\sum_{j=1}^{N_h} A_{i j}\geq \sum_{j=1}^{N_h+N_h^{\partial}} A_{i j}$ due to $A_{i j} \leq 0$ for $i\neq j$, 
we do have \eqref{weak-diag-dominance} holds. When $\mathbf{x}_i$ has two neighboring node on the boundary, based on \eqref{local-stif-matr-val}, among the two neighboring nodes on the boundary of $\mathbf{x}_i$, there exists nodes $\mathbf{x}_l$ with $l\in \{N_h+1, \dots, N_h+N_h^{\partial}\}$ such that $A_{i l} < 0$. Then we have
\begin{equation*}
\sum_{j=1}^{N_h} A_{i j} \geq \sum_{j=1}^{N_h+N_h^{\partial}} A_{i j} - A_{i l} > 0.
\end{equation*}
Therefore, the stiffness matrix $A$ is an $M$-matrix. Since the lumped mass matrix is diagonal and entry-wise positive, with $A_{i j} \leq 0$ and noticing that $\bar{L}_h^{-1}=\left(\begin{array}{cc}L_h^{-1} & -L_h^{-1} B^{\partial} \\ 0 & I\end{array}\right)$, we conclude  $\bar{L}_h$ is also an $M$-matrix. Then with \eqref{stiff-row-sum} and Theorem \ref{them-dmp} we obtain the finite difference operator defined by $\bar{L}_h$ satisfies discrete maximum principle.
\end{proof}

\begin{rmk}
For each element $e$, the choice in \eqref{param-setting} make $\lambda_e^1, \lambda_e^2 > 0$, which implies the $V^h$-ellipticity of the bilinear form \eqref{bilinear-form-appr} discussed in Section \ref{sec-vh-ellipticity}. Therefore, we can assure of $V^h$-ellipticity and the stiffness matrix being an M-matrix simultaneously.
\end{rmk}
\begin{rmk}
The constraint on the coefficient, $|\bar{a}^{12}_e| \leq \min\{\bar{a}^{11}_e,\bar{a}^{22}_e\}$, aligns with the condition for rendering the stiffness matrix as an $M$-matrix in the seven-point stencil control volume method with optimal optimal monotonicity region in the case of homogeneous medium and uniform mesh in \cite{nordbotten2007monotonicity}. In \cite{ngo2016monotone}, the authors show that a three-by-three stencil can be used to construct monotone finite difference schemes under the assumption $|a^{12}| < \min\{a^{11},a^{22}\}$.
\end{rmk}
\begin{rmk}
If the domain is not convex, e.g., an $L$-shaped domain, as long as it can be partitioned by uniform square meshes satisfying the coefficients constraints or quadrilateral meshes satisfying the mesh constraints derived in Section \ref{sec-gen-quad-mesh}, the stiffness matrix is still an $M$-matrix and the monotonicity holds. But the a priori error estimates in Section \ref{sec-convergence} might no longer hold due to possible loss of the elliptic regularity on a nonconvex domain.
\end{rmk}

\begin{rmk}
The choice of the quadrature parameters in \eqref{param-setting} is sharp to enforce the $\bar{L}_h$ being an $M$-matrix but not for monotonicity since $M$-matrix property is just a sufficient but not necessary condition for monotonicity.
\end{rmk}

\section{Convergence of the $Q^1$ finite element method with mixed quadrature}\label{sec-convergence}

In this section, we prove the second-order accuracy of the scheme \eqref{hom-var-num-quad} on uniform rectangular mesh. For simplicity we only prove result for the problem with homogeneous Dirichlet boundary condition, i.e. $g=0$. For convenience, in this section, we may drop the subscript $h$ in a test function $v_h \in V^h$.  When there is no confusion, we may also drop $d\mathbf{x}$ or $d\hat{\mathbf{x}}$ in a integral.  

\subsection{Approximation error estimate of bilinear forms}
In this subsection, we estimate the approximation error of $\mathcal A_h(u,v)$ to $\mathcal A(u,v)$.

\begin{theorem}\label{bilinear-form-appr-error}
Assume $a^{ij}, c \in W^{2, \infty}(\Omega)$ for $i,j = 1,2$ and $u \in H^3(\Omega)$, then $\forall v \in V^h$, on element $e$, we have
\begin{align}
\int_e\left(\mathbf{a}\nabla u\right)\cdot \nabla v d\mathbf{x} - \int_e\left(\bar{\mathbf{a}}_e\nabla u\right) \cdot \nabla v d^h_{\boldsymbol{\lambda}_e}\mathbf{x}=& \mathcal{O}(h^2)\|u\|_{3,e}\|v\|_{2,e},\\
\int_e c uv d\mathbf{x}  - \int_e cuv d^h_{1}\mathbf{x}  = &\mathcal{O}\left(h^{2}\right) \|u\|_{2, e}\|v\|_{2, e}.
\end{align}
\end{theorem}
\begin{proof}
For $k, l = 1,2$ and function $a\in W^{2, \infty}(e)$, we have
\begin{equation}\label{bilinear-quad-err}
\begin{aligned}
& \int_e a u_{x_k}v_{x_l}d\mathbf{x} -\int_e \bar a_e u_{x_k}v_{x_l}d^h_{\boldsymbol{\lambda}_e}\mathbf{x}\\ 
=& \int_e (a-\bar{a}_e) u_{x_k} v_{x_l}d\mathbf{x} 
+\bar a_e \left(\int_e u_{x_k}v_{x_l}d\mathbf{x} - \int_e u_{x_k}v_{x_l}d^h_{\boldsymbol{\lambda}_e}\mathbf{x} \right) \\ 
=& \int_e (a-\bar{a}_e) u_{x_k} v_{x_l}d\mathbf{x} + \bar a_eE( u_{x_k}v_{x_l}).
\end{aligned}
\end{equation}

For the first term,
\begin{equation}
\begin{aligned}
&\int_e (a-\bar{a}_e) u_{x_k} v_{x_l}d\mathbf{x}\\
= & \int_e (a-\bar a_e) (u_{x_k} v_{x_l}-\overline{u_{x_k} v_{x_l}}) d \mathbf{x}  + \int_e (a-\bar a_e) \overline{u_{x_k} v_{x_l}} d \mathbf{x}\\
\leq & \|a-\bar a_e\|_{0, \infty, e} \|u_{x_k} v_{x_l}-\overline{u_{x_k} v_{x_l}}\|_{0, 1, e} + \frac{1}{meas(e)}\int_e (a-\bar a_e)d \mathbf{x} \int_e u_{x_k} v_{x_l}d \mathbf{x}.
\end{aligned}
\end{equation}
By Poincare inequality and Cauchy-Schwartz inequality, we have
\begin{equation}
\begin{aligned}
& \|a-\bar a_e\|_{0, \infty, e} \|u_{x_k} v_{x_l}-\overline{u_{x_k} v_{x_l}}\|_{0, 1, e} \\
=&\mathcal{O}(h^2)\|a\|_{1, \infty, e}\left\|\nabla\left(u_{x_k} v_{x_l}\right)\right\|_{0,1, e} =\mathcal{O}(h^2)\|u\|_{2, e}\|v\|_{2, e}.
\end{aligned}
\end{equation}
By Lemma \ref{quaderror-theorem} and Cauchy-Schwartz inequality 
\begin{equation}
\begin{aligned}
& \frac{1}{meas(e)}\int_e (a-\bar a_e)d \mathbf{x} \int_e u_{x_k} v_{x_l}d \mathbf{x} \\
= & \frac{h^{2+d}}{meas(e)}[a]_{2, \infty, e}\|u_{x_k}\|_{0, e}\|v_{x_l}\|_{0, e} = \mathcal{O}\left(h^{2}\right)\|u\|_{1, e}\|v\|_{1, e}
\end{aligned}
\end{equation}
where in the last equation $meas(e)=\mathcal{O}(h^d)$ is also used.
Therefore, we have the estimate of the first term of \eqref{bilinear-quad-err}:
\begin{equation}\label{1st-term-estimate}
\int_e (a-\bar a_e) u_{x_k} v_{x_l} d \mathbf{x} = \mathcal{O}(h^2)\|a\|_{2, \infty, e}\|u\|_{2, e}\|v\|_{2, e}.
\end{equation}

For the second term of \eqref{bilinear-quad-err}, by Lemma \ref{2nd-order-bh-cor}, we obtain
\begin{equation}
\int_e \bar a_e u_{x_k} v_{x_l} d\mathbf{x}  -\int_e \bar a_e u_{x_l} v_{x_l} d^h_{\boldsymbol{\lambda}_e}\mathbf{x} = \mathcal{O}(h^2)\|a\|_{0, \infty, e}\|u\|_{3, e}\|v\|_{2, e},
\end{equation}
which together with \eqref{1st-term-estimate} imply the estimate of \eqref{bilinear-quad-err}:
\begin{equation}
\int_e a u_{x_k}v_{x_l}d\mathbf{x} -\int_e \bar a_e u_{x_k}v_{x_l}d^h_{\boldsymbol{\lambda}_e}\mathbf{x} = \mathcal{O}(h^2)\|a\|_{2, \infty, e}\|u\|_{3, e}\|v\|_{2, e}.
\end{equation}
Therefore, we have
\begin{equation}
\int_e \left(\mathbf{a}(\mathbf{x}) \nabla u\right) \cdot\nabla v d\mathbf{x} - \int_e \left(\bar{\mathbf{a}}(\mathbf{x}) \cdot \nabla u\right) \nabla v d^h_{\boldsymbol{\lambda}_e}\mathbf{x} = \mathcal{O}\left(h^{2}\right)\|\mathbf{a}\|_{2, \infty, e} \|u\|_{3, e}\|v\|_{2, e}.
\end{equation}
Similarly we have
\begin{equation}
\int_e c u v d\mathbf{x} - \int_e cuv d^h_{1}\mathbf{x} = \mathcal{O}\left(h^{2}\right) \|c\|_{2, \infty, e}\|u\|_{2, e}\|v\|_{2, e}.
\end{equation}

\end{proof}

We also have
\begin{lemma}\label{1st-order-bh-cor-bilinear}
Assume $a^{ij}, c \in W^{2, \infty}(\Omega)$ for $i,j = 1,2$. We have
$$
A\left(v_h, w_h\right)-A_h\left(v_h, w_h\right)=\mathcal{O}(h)\left\|v_h\right\|_2\left\|w_h\right\|_1, \quad \forall v_h, w_h \in V^h
$$
\end{lemma}
\begin{proof}
By Theorem \ref{bilinear-form-appr-error} and noticing that the third derivative of $Q^1$ polynomial vanish, we have
\begin{align}
\int_e\left(\mathbf{a}\nabla v_h\right)\cdot \nabla w_h d\mathbf{x} - \int_e\left(\bar{\mathbf{a}}_e\nabla v_h\right) \cdot \nabla w_h d^h_{\boldsymbol{\lambda}_e}\mathbf{x}=& \mathcal{O}(h^2)\|v_h\|_{2,e}\|w_h\|_{2,e},\\
\int_e c v_h w_h d\mathbf{x}  - \int_e c v_h w_h d^h_{1}\mathbf{x}  = &\mathcal{O}\left(h^{2}\right) \|v_h\|_{2, e}\|w_h\|_{2, e}.
\end{align}
By applying the inverse estimate to polynomial $z_h$, we obtain
\begin{align}
\int_e\left(\mathbf{a}\nabla v_h\right)\cdot \nabla w_h d\mathbf{x} - \int_e\left(\bar{\mathbf{a}}_e\nabla v_h\right) \cdot \nabla w_h d^h_{\boldsymbol{\lambda}_e}\mathbf{x}=& \mathcal{O}(h)\|v_h\|_{2,e}\|w_h\|_{1,e},\\
\int_e c v_h w_h d\mathbf{x}  - \int_e c v_h w_h d^h_{1}\mathbf{x}  = &\mathcal{O}\left(h\right) \|v_h\|_{2, e}\|w_h\|_{1, e}.
\end{align}
Then by summing over all the elements we obtain prove the Lemma.
\end{proof}

\subsection{$V^h$-ellipticity and the dual problem}\label{sec-vh-ellipticity}
In order to prove the convergence results of the scheme \eqref{hom-var-num-quad}, we need $A_h$ satisfies $V^h$-ellipticity:
\begin{equation}\label{vh-ellipticity}
\forall v_h \in V_0^h, \quad C\left\|v_h\right\|_1^2 \leq A_h\left(v_h, v_h\right).  
\end{equation}

By following the proof of Lemma 5.1 in \cite{li2020superconvergence}, we have
\begin{lemma}\label{vh-ellipticity-lemma}
Assume the eigenvalues of $\mathbf{a}$ have a uniform positive lower bound and a uniform upper bound and $c$ have a upper bound. If there exists lower bound $\lambda_0 >0$ such that $\forall e \in \Omega_h$, the quadrature parameter $\lambda_e^1, \lambda_e^2 > \lambda_0$, then there are two constants $C_1, C_2>0$ independent of mesh size $h$ such that
$$
\forall v_h \in V_0^h, \quad C_1\left\|v_h\right\|_1^2 \leq A_h\left(v_h, v_h\right) \leq C_2\left\|v_h\right\|_1^2 .
$$
\end{lemma}
\begin{proof}
For element $e$, at first we map all the functions to the reference element $\hat{K}$. Let $Z_{0, \hat{K}}$ denote the set of vertices on the reference element $\hat{K}$. We notice that the set $Z_{0, \hat{K}}$ is a $Q^1(\hat{K})$-unisolvent subset. Since the weights of trapezoid rule are strictly positive, we have
$$
\forall \hat{p} \in Q^1(\hat{K}), \quad \sum_{i=1}^2\int_{\hat{K}} \hat{p}_{\hat x^i}^2 d^h_{1}\hat{\mathbf{x}} =0 \Longrightarrow \hat{p}_{\hat x^i}=0 \text { at } Z_{0, \hat{K}},
$$
where $i=1,2$. As a consequence, $\sum_{i=1}^2\int_{\hat{K}} \hat{p}_{\hat x^i}^2 d^h_{1}\hat{\mathbf{x}}$ defines a norm over the quotient space $Q^1(\hat{K}) / Q^0(\hat{K})$. Since that $|\cdot|_{1, \hat{K}}$ is also a norm over the same quotient space, by the equivalence of norms over a finite dimensional space, we have
$$
\forall \hat{p} \in Q^1(\hat{K}), \quad C_1|\hat{p}|_{1, \hat{K}}^2 \leq \sum_{i=1}^2\int_{\hat{K}} \hat{p}_{\hat x^i}^2 d^h_{1}\hat{\mathbf{x}} \leq C_2|\hat{p}|_{1, \hat{K}}^2
$$
As the quadrature parameter $\lambda_e^1, \lambda_e^2 \geq \lambda_0 \geq 0$, we have
$$
C_1\left|\hat{v}_h\right|_{1, \hat{K}}^2 \leq C_1 \sum_{i=1}^2\int_{\hat{K}} (\hat{v}_h)_{\hat x_i}^2 d^h_{1}\hat{\mathbf{x}}  \leq 
\int_{\hat{K}} (\bar{\mathbf{a}}^{i j}_e \nabla \hat{v}_h)\cdot \nabla \hat{v}_h d^h_{\boldsymbol{\lambda_e}}\hat{\mathbf{x}} + \int_{\hat{K}}  \hat{c} \hat{v}_h^2 d^h_{1}\hat{\mathbf{x}} \leq C_2\left\|\hat{v}_h\right\|_{1, \hat{K}}^2.
$$
Mapping these back to the original element $e$ and summing over all elements, by the equivalence of two norms $|\cdot|_1$ and $\|\cdot\|_1$ for the space $H_0^1(\Omega) \supset V_0^h$, we obtain the conclusion.
\end{proof}

In the following part, we assume the assumption of Lemma \ref{vh-ellipticity-lemma} is fulfilled, i.e. the $V^h$-ellipticity holds.

In order to apply the Aubin-Nitsche duality argument for establishing convergence of function values, we need certain estimates on a proper dual problem. 

Define $\theta:=u-u_h$ and consider the dual problem: find $w \in H_0^1(\Omega)$ satisfying
\begin{equation}\label{dual-problem}
A^*(w, v)=\left(\theta, v\right), \quad \forall v \in H_0^1(\Omega),
\end{equation}
where $A^*(\cdot, \cdot)$ is the adjoint bilinear form of $A(\cdot, \cdot)$ such that
$$
A^*(u, v)=A(v, u)=(\mathbf{a} \nabla v, \nabla u)+(c v, u) .
$$
Although here the bilinear form we considered is symmetric i.e. $A(\cdot, \cdot) = A^*(\cdot, \cdot)$, we still use $A^*(\cdot, \cdot)$ for abstractness.

Let $w_h \in V_0^h$ be the solution to
\begin{equation}\label{dual-problem-num}
A_h^*\left(w_h, v_h\right)=\left(\theta, v_h\right), \quad \forall v_h \in V_0^h.
\end{equation}

Notice that the right hand side of \eqref{dual-problem-num} is different from the right hand side of the scheme \eqref{hom-var-num-quad}.

We have the following standard estimates on $w_h$ for the dual problem.
\begin{lemma}\label{lem-w-wh-1-norm}
 Assume $a^{i j}, c \in W^{2, \infty}(\Omega)$ and $u \in H^3(\Omega), f\in$ $H^{2}(\Omega)$. Let $w$ be defined in \eqref{dual-problem}, $w_h$ be defined in \eqref{dual-problem-num}. With elliptic regularity and $V^h$-ellipticity hold, we have
 \begin{equation}
\begin{aligned}
\left\|w-w_h\right\|_1 \leq & C h\|w\|_2 \\
\left\|w_h\right\|_2 \leq & C\left\|\theta\right\|_0 .
\end{aligned}
\end{equation}
\end{lemma}
\begin{proof}
By $V^h$-ellipticity, we have $C_1\left\|w_h-v_h\right\|_1^2 \leq A_h^*\left(w_h-v_h, w_h-v_h\right)$. By the definition of the dual problem \eqref{dual-problem}, we have
$$
A_h^*\left(w_h, w_h-v_h\right)=\left(\theta, w_h-v_h\right)=A^*\left(w, w_h-v_h\right), \quad \forall v_h \in V_0^h.
$$
Therefore  $\forall v_h \in V_0^h$, by Lemma \ref{1st-order-bh-cor-bilinear}, we have
$$
\begin{aligned}
& C_1\left\|w_h-v_h\right\|_1^2 \leq A_h^*\left(w_h-v_h, w_h-v_h\right) \\
= & A^*\left(w-v_h, w_h-v_h\right)+\left[A_h^*\left(w_h, w_h-v_h\right)-A^*\left(w, w_h-v_h\right)\right]+\left[A^*\left(v_h, w_h-v_h\right)-A_h^*\left(v_h, w_h-v_h\right)\right] \\
= & A^*\left(w-v_h, w_h-v_h\right)+\left[A\left(w_h-v_h, v_h\right)-A_h\left(w_h-v_h, v_h\right)\right] \\
\leq & C\left\|w-v_h\right\|_1\left\|w_h-v_h\right\|_1+C h\left\|v_h\right\|_2\left\|w_h-v_h\right\|_1,
\end{aligned}
$$
which implies
\begin{equation}\label{w-wh-1norm}
\quad\left\|w-w_h\right\|_1 \leq\left\|w-v_h\right\|_1+\left\|w_h-v_h\right\|_1 \leq C\left\|w-v_h\right\|_1+C h\left\|v_h\right\|_2.
\end{equation}

Now consider $\Pi_1 w \in V_0^h$ where $\Pi_1$ is the piece-wise $Q^1$ projection and its definition on each element is defined through \eqref{q1-proj} on the reference element. By Theorem \ref{bh-lemma} on the projection error, we have
\begin{equation}\label{w-proj-w}
\left\|w-\Pi_1 w\right\|_1 \leq C h\|w\|_2, \quad\left\|w-\Pi_1 w\right\|_2 \leq C\|w\|_2,
\end{equation}
which implies
\begin{equation}\label{proj-w-2-norm}
\left\|\Pi_1 w\right\|_2 \leq\|w\|_2+\left\|w-\Pi_1 w\right\|_2 \leq C\|w\|_2.
\end{equation}
By setting $v_h=\Pi_1 w$, using \eqref{w-wh-1norm}, \eqref{w-proj-w} and \eqref{proj-w-2-norm}, we have
\begin{equation}\label{w-wh-1norm-2}
\left\|w-w_h\right\|_1 \leq C\left\|w-\Pi_1 w\right\|_1+C h\left\|\Pi_1 w\right\|_2 \leq C h\|w\|_2 .
\end{equation}

By \eqref{w-proj-w} and \eqref{w-wh-1norm-2}, we also have
\begin{equation}\label{w-wh-1norm-3}
\left\|w_h-\Pi_1 w\right\|_1 \leq\left\|w-\Pi_1 w\right\|_1+\left\|w-w_h\right\|_1 \leq C h\|w\|_2 .
\end{equation}

By the inverse estimate on the piece-wise polynomial $w_h-\Pi_1 w$, we obtain
\begin{equation}\label{wh-2}
\left\|w_h\right\|_2 \leq\left\|w_h-\Pi_1 w\right\|_2+\left\|\Pi_1 w-w\right\|_2+\|w\|_2 \leq C h^{-1}\left\|w_h-\Pi_1 w\right\|_1+C\|w\|_2.
\end{equation}
With \eqref{w-wh-1norm-3}, \eqref{wh-2} and the elliptic regularity $\|w\|_2 \leq C\left\|\theta\right\|_0$, we obtain
$$
\left\|w_h\right\|_2 \leq C\|w\|_2 \leq C\left\|\theta\right\|_0 .
$$
\end{proof}

\subsection{Convergence results}

In this section, we initially establish the error estimate for $\left\|u-u_h\right\|_{1, \Omega}$. Subsequently, we demonstrate that the $Q^1$ finite element method, as given by \eqref{hom-var-num-quad}, achieves second-order accuracy for function values.

We have the estimate of the error $\left\|u-u_h\right\|_{1,\Omega}$ as follows:
\begin{theorem}\label{thm-u-uh-1-norm}
 Assume $a^{i j}, c \in W^{2, \infty}(\Omega)$ and $u\in H^{2}(\Omega), f \in$ $H^{2}(\Omega)$. With elliptic regularity and $V^h$-ellipticity hold, we have
$$
\left\|u-u_h\right\|_{1, \Omega}=\mathcal{O}\left(h\right)\left(\|u\|_{2, \Omega}+\|f\|_{2, \Omega}\right).
$$
\end{theorem}
\begin{proof}
By the First Strang Lemma,
\begin{equation}
\begin{aligned}
\left\|u-u_h\right\|_{1, \Omega} \leq & C\left(\inf _{v_h \in V^h}\left\{\left\|u-v_h\right\|_{1, \Omega}+\sup _{w_h \in V_h} \frac{\left|\mathcal{A}\left(v_h, w_h\right)-\mathcal{A}_h\left(v_h, w_h\right)\right|}{\left\|w_h\right\|_{1, \Omega}}\right\}+\right. \\
& \left.\quad+\sup _{w_h \in V^h} \frac{\left|\left\langle f, w_h\right\rangle_h-\left(f, w_h\right)\right|}{\left\|w_h\right\|_{1, \Omega}}\right).
\end{aligned}
\end{equation}

By Lemma \ref{1st-order-bh-cor-bilinear}, we have:
\begin{equation*}
\begin{aligned}
\frac{\left|\mathcal{A}\left(v_h, w_h\right)-\mathcal{A}_h\left(v_h, w_h\right)\right|}{\left\|w_h\right\|_{1, \Omega}}
= \frac{\mathcal{O}(h)\|v_h\|_{2,\Omega}\|w_h\|_{1,\Omega}}{\left\|w_h\right\|_{1, \Omega}}
= \mathcal{O}(h)\|v_h\|_{2,\Omega}.
\end{aligned}
\end{equation*}

By Lemma \ref{rhs-quad-error}, we have 
$$
\sup _{w_h \in V^h} \frac{\left|\left\langle f, w_h\right\rangle_h-\left(f, w_h\right)\right|}{\left\|w_h\right\|_{1, \Omega}} = \frac{\mathcal{O}(h^2)\|f\|_{2,\Omega}\|w_h\|_{1,\Omega}}{\|w_h\|_{1,\Omega}} = \mathcal{O}(h^2)\|f\|_{2,\Omega}.
$$

By the approximation property of piece-wise $Q^1$ polynomials,
\begin{equation*}
\begin{aligned}
\left\|u-u_h\right\|_{1, \Omega} = \mathcal{O}(h)(\|u\|_{2,\Omega}+ |f\|_{2,\Omega}).
\end{aligned}
\end{equation*}

\end{proof}

In the following part we prove the Aubin-Nitsche Lemma up to the quadrature error for establishing convergence of function values.
\begin{theorem}
 Assume $a^{i j}, c \in W^{2, \infty}(\Omega)$ and $u(\mathbf{x}) \in H^3(\Omega), f\in$ $H^{2}(\Omega)$.  Assume $V^h$ ellipticity holds. Then the numerical solution from scheme \eqref{hom-var-num-quad} $u_h$ is a $2$-th order accurate approximation to the exact solution $u$:
$$
\left\|u_h-u\right\|_{0, \Omega}=\mathcal{O}\left(h^2\right)\left(\|u\|_{2, \Omega}+\|f\|_{2, \Omega}\right).
$$
\end{theorem}
\begin{proof}
With $\theta=u-u_h \in H_0^1(\Omega)$, we have
\begin{equation}\label{dual-main-split}
\begin{aligned}
\left\|\theta\right\|_0^2=\left(\theta, \theta\right)=A\left(\theta,  w\right)=A\left(u-u_h, w_h\right) + A\left(u-u_h, w-w_h\right) 
\end{aligned}
\end{equation}
For the first term \eqref{dual-main-split}, by Lemma \ref{bilinear-form-appr-error}, we have
\begin{equation}\label{dual-main-split-1-term}
\begin{aligned}
& A\left(u-u_h, w_h\right)=\left[A\left(u, w_h\right)-A_h\left(u_h, w_h\right)\right]+\left[A_h\left(u_h, w_h\right)-A\left(u_h, w_h\right)\right]\\
= & \left(f, w_h\right)-\left\langle f, w_h\right\rangle_h +\mathcal{O}\left(h^2 \right)\|u_h\|_{3}\left\|w_h\right\|_2 \\
= &\mathcal{O}\left(h^{2}\right)\|f\|_{2}\left\|w_h\right\|_1 + \mathcal{O}\left(h^{2}\right)\|u_h\|_{2}\left\|w_h\right\|_2 \\
= & \mathcal{O}\left(h^{2}\right)(\|f\|_{2} +  \|u_h\|_{2})\|\theta\|_{0},
\end{aligned}
\end{equation}
where in the second last equation Lemma \ref{rhs-quad-error} and the fact the third derivative of $Q^1$ polynomials vanish are used.
As the estimate of $\|w_h\|_2$ and $\|w\|_2$ in the proof of Lemma \ref{lem-w-wh-1-norm}, we have
\begin{equation}\label{uh-2}
\begin{aligned}
\left\|u_h\right\|_2 \leq & \left\|u_h-\Pi_1 u\right\|_2+\left\|\Pi_1 u - u\right\|_2+\|u\|_2 \leq C h^{-1}\left\|u_h-\Pi_1 u\right\|_1+C\|u\|_2 \\
\leq & C h^{-1}\left(\|u-\Pi_1 u\|_1+\left\|u-u_h\right\|_1\right) + C\|u\|_2 \\
\leq & C h^{-1}\left\|u-u_h\right\|_1 + C\|u\|_2 \\
\leq & C(\|u\|_2 + \|f\|_2),
\end{aligned}
\end{equation}
where Theorem \ref{thm-u-uh-1-norm} is used in the last inequality. Therefore, we have
\begin{equation}
\begin{aligned}
A\left(u-u_h, w_h\right) = \mathcal{O}\left(h^{2}\right)(\|f\|_{2} +  \|u\|_{2})\|\theta\|_{0}.
\end{aligned}
\end{equation}

For the second term \eqref{dual-main-split}, by continuity of the bilinear form and Lemma \ref{lem-w-wh-1-norm}, we have
\begin{equation}\label{dual-main-split-2-term}
\begin{aligned}
& A\left(u-u_h, w-w_h\right) \leq C\left\|u-u_h\right\|_1\left\|w-w_h\right\|_1 \leq Ch\left\|u-u_h\right\|_1\|w\|_2 \\
\leq & Ch\left\|u-u_h\right\|_1 \|\theta\|_{0} = \mathcal{O}\left(h^{2}\right)(\|f\|_{2} +  \|u\|_{2})\|\theta\|_{0}.
\end{aligned}
\end{equation}
Therefore, by \eqref{dual-main-split}, \eqref{dual-main-split-1-term} and \eqref{dual-main-split-2-term}, we have 
\begin{equation}
\begin{aligned}
\left\|\theta\right\|_0  = \mathcal{O}\left(h^{2}\right)(\|f\|_{2} +  \|u\|_{2}).
\end{aligned}
\end{equation}

\end{proof}

\begin{rmk}
Similar convergence results for the $Q^1$ method on general quasi-uniform quadrilateral meshes can be established via the same proof procedure in this section.
\end{rmk}



\section{Extension to general quadrilateral meshes}\label{sec-gen-quad-mesh}
For a quadrilateral element $e$ as in Fig. \ref{quadrilateral-element}, let $\mathbf{F}_e = (F_{e1}, F_{e2})^T$ be the mapping such that $\mathbf{F}_e(\hat{K})=e$.

For $\varphi \in V^h_0$, by definition $\hat{\varphi} = \varphi|_e \circ \mathbf{F}_e \in Q^1(\hat K)$.
According to the chain rule, we have 
$$
\nabla \varphi \circ \mathbf{F}_e=D F_e^{T-1} \hat{\nabla} \hat{\varphi}
$$
where $\hat{\varphi} = \varphi \circ \mathbf{F}_e$, $\nabla=\left(\frac{\partial \quad}{\partial x_1}, \frac{\partial \quad}{\partial x_2}\right)^T$, $\hat{\nabla}=\left(\frac{\partial \quad }{\partial \hat{x}_1}, \frac{\partial \quad}{\partial \hat{x}_2}\right)^T$ and Jacobian matrix $D F_e=\left(\begin{array}{ll}\frac{\partial F_{e 1}}{\partial \hat{x}_1} & \frac{\partial F_{e 1}}{\partial \hat{x}_2} \\ \frac{\partial F_{e 2}}{\partial \hat{x}_1} & \frac{\partial F_{e 2}}{\partial \hat{x}_2}\end{array}\right)$.

Therefore, we have
\begin{equation}\label{int-trans}
\int_{e} \mathbf{a} \nabla u_h \cdot \nabla v_h d \mathbf{x} = \int_{\hat{K}} \left(DF_e^{-1}\hat{\mathbf{a}} DF_e^{T-1} \hat{\nabla} \hat{u}_h\right) \cdot \hat{\nabla} \hat{v}_h \left|\operatorname{det}(DF_e)\right| d\hat{\mathbf{x}}.
\end{equation}

In the case of regular meshes with mesh size $h$, the matrix $D F_e^{-1} \hat{\mathbf{a}}D F_e^{T-1} = \frac{1}{h^2}\hat{\mathbf{a}}$ and $\operatorname{det}(DF_e) = h^2$.

Approximate \eqref{int-trans} by the mixed quadrature \eqref{mixed-quad-ref} with parameter $\boldsymbol{\lambda} = (\lambda^1, \lambda^2)$, i.e.
\begin{equation}\label{int-trans-quad}
\int_{e} \left(\mathbf{a} \nabla u_h \right) \cdot \nabla v_h \mathrm{~d} \mathbf{x} \approx \int_{\hat{K}} \left( \tilde{\mathbf{a}} \hat{\nabla} \hat{u}_h\right) \cdot \hat{\nabla} \hat{v}_h d^h_{\boldsymbol{\lambda}}\hat{\mathbf{x}}
\end{equation}
where $\tilde{\mathbf{a}} =  \left(\left|\operatorname{det}(DF_e)\right|DF_e^{-1}\hat{\mathbf{a}} DF_e^{T-1}\right)(\frac{1}{2}, \frac{1}{2})$.

As in Fig. \ref{quadrilateral-element}, denote $$\overrightarrow{\mathbf{c}_0} = \mathbf{c}_{0,1}-\mathbf{c}_{0,0}, \quad \overrightarrow{\mathbf{c}_1} = \mathbf{c}_{1,0}-\mathbf{c}_{0,0}, \quad \overrightarrow{\mathbf{c}_2} = \mathbf{c}_{1,1}-\mathbf{c}_{1,0}, \quad \overrightarrow{\mathbf{c}_3} = \mathbf{c}_{1,1}-\mathbf{c}_{0,1}$$
and
\[
\overrightarrow{\mathbf{c}_i} = (c_i^0, c_i^1)^T, \, i = 0,1,2,3,\quad DF_h = DF_e(\frac{1}{2}, \frac{1}{2}), \quad \bar{\mathbf{a}}_e = \mathbf{a}|_e\circ \mathbf{F}_e(\frac{1}{2}, \frac{1}{2}),
\]
then we have 
\begin{align*}
DF_h = \frac{1}{2}
\begin{pmatrix}
c_1^0+c_3^0 & c_0^0+c_2^0\\
c_1^1+c_3^1 & c_0^1+c_2^1\\
\end{pmatrix}, \quad 
DF_h^{-1} = \frac{1}{2\operatorname{det}(DF_h)}
\begin{pmatrix}
c_0^1+c_2^1 & -c_0^0-c_2^0\\
-c_1^1-c_3^1 & c_1^0+c_3^0\\
\end{pmatrix},\\
\end{align*}
and on element $e$, we have
\begin{equation}\label{tilde-a-def}
\tilde{\mathbf{a}} =  \operatorname{det}(DF_h)DF_h^{-1}\bar{\mathbf{a}}_e DF_h^{T-1} = 
\begin{pmatrix}
\tilde{a}_e^{11} & \tilde{a}_e^{12}\\
\tilde{a}_e^{12} & \tilde{a}_e^{22}\\
\end{pmatrix}.
\end{equation}

To have the stiffness matrix an $M$-matrix, by Theorem \ref{thm-stif-m-matr}, the following is a sufficient condition:
\begin{equation}\label{quad-mesh-m-matr-cond}
\left|\tilde{a}_e^{12}\right| \leq \min\{\tilde{a}_e^{11}, \tilde{a}_e^{22} \},
\end{equation}
where
\begin{equation*}
\begin{aligned}
\tilde{a}^{11} 
= & \frac{1}{4|\operatorname{det}(DF_h)|}\begin{pmatrix}
c_0^1+c_2^1 & -c_0^0-c_2^0
\end{pmatrix}
\begin{pmatrix}
\bar{a}^{11} & \bar{a}^{12} \\
\bar{a}^{12} & \bar{a}^{22} \\
\end{pmatrix}
\begin{pmatrix}
c_0^1+c_2^1 \\
-c_0^0-c_2^0
\end{pmatrix}\\
= & \frac{1}{4|\operatorname{det}(DF_h)|}\begin{pmatrix}
c_0^0+c_2^0 & c_0^1+c_2^1
\end{pmatrix}
\begin{pmatrix}
0 & -1 \\
1 & 0 \\
\end{pmatrix}
\begin{pmatrix}
\bar{a}^{11} & \bar{a}^{12} \\
\bar{a}^{12} & \bar{a}^{22} \\
\end{pmatrix}
\begin{pmatrix}
0 & 1 \\
-1 & 0 \\
\end{pmatrix}
\begin{pmatrix}
c_0^0+c_2^0 \\
c_0^1+c_2^1
\end{pmatrix}\\
= & \frac{1}{4|\operatorname{det}(DF_h)|}\begin{pmatrix}
c_0^0+c_2^0 & c_0^1+c_2^1
\end{pmatrix}
\begin{pmatrix}
\bar{a}^{22} & -\bar{a}^{12} \\
-\bar{a}^{12} & \bar{a}^{11} \\
\end{pmatrix}
\begin{pmatrix}
c_0^0+c_2^0 \\
c_0^1+c_2^1
\end{pmatrix}\\
= & \frac{\operatorname{det}(\bar{\mathbf{a}}_e)}{4|\operatorname{det}(DF_h)|}\left( \overrightarrow{\mathbf{c}_0} + \overrightarrow{\mathbf{c}_2}\right)^T\bar{\mathbf{a}}_e^{-1} \left( \overrightarrow{\mathbf{c}_0} + \overrightarrow{\mathbf{c}_2}\right),
\end{aligned}
\end{equation*}
and similarly
\begin{align*}
\tilde{a}^{12} = &  -\frac{\operatorname{det}(\bar{\mathbf{a}}_e)}{4|\operatorname{det}(DF_h)|}\left( \overrightarrow{\mathbf{c}_0} + \overrightarrow{\mathbf{c}_2}\right)^T\bar{\mathbf{a}}_e^{-1} \left( \overrightarrow{\mathbf{c}_1} + \overrightarrow{\mathbf{c}_3}\right), \\
\tilde{a}^{22} = & \frac{\operatorname{det}(\bar{\mathbf{a}}_e)}{4|\operatorname{det}(DF_h)|}\left( \overrightarrow{\mathbf{c}_1} + \overrightarrow{\mathbf{c}_3}\right)^T\bar{\mathbf{a}}_e^{-1} \left( \overrightarrow{\mathbf{c}_1} + \overrightarrow{\mathbf{c}_3}\right).
\end{align*}

By $\overrightarrow{\mathbf{c}_1}+\overrightarrow{\mathbf{c}_2}-\overrightarrow{\mathbf{c}_3} - \overrightarrow{\mathbf{c}_0} = \overrightarrow{0}$, we note
\eqref{quad-mesh-m-matr-cond} is equivalent to
\begin{equation}\label{mesh-cond}
\begin{aligned}
\left( \overrightarrow{\mathbf{c}_0} + \overrightarrow{\mathbf{c}_2}\right)^T\bar{\mathbf{a}}_e^{-1}\left( \overrightarrow{\mathbf{c}_0} + \overrightarrow{\mathbf{c}_3}\right) \geq & 0, \quad
\left( \overrightarrow{\mathbf{c}_0} + \overrightarrow{\mathbf{c}_2}\right)^T\bar{\mathbf{a}}_e^{-1}\left( \overrightarrow{\mathbf{c}_0} - \overrightarrow{\mathbf{c}_1}\right) \geq  0, \\
\left( \overrightarrow{\mathbf{c}_1} + \overrightarrow{\mathbf{c}_3}\right)^T\bar{\mathbf{a}}_e^{-1}\left( \overrightarrow{\mathbf{c}_0} + \overrightarrow{\mathbf{c}_3}\right) \geq & 0, \quad
\left( \overrightarrow{\mathbf{c}_1} + \overrightarrow{\mathbf{c}_3}\right)^T\bar{\mathbf{a}}_e^{-1}\left( \overrightarrow{\mathbf{c}_1} - \overrightarrow{\mathbf{c}_0}\right) \geq  0.
\end{aligned}
\end{equation}

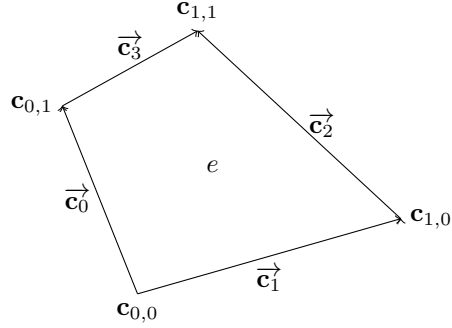
\begin{figure}[h]\label{quadrilateral-element}
\centering
\begin{tikzpicture}
\coordinate (A) at (0,0);
\coordinate (B) at (3.5,1);
\coordinate (C) at (.8,3.5);
\coordinate (D) at (-1,2.5);

\coordinate (AB) at (1.75,0.5);
\coordinate (BC) at (2.15,2.25);
\coordinate (CD) at (-0.1,3);
\coordinate (AD) at (-0.5,1.25);

\coordinate (E) at (1,1.5);

\path
(A) node[below] {$\mathbf{c}_{0,0}$}
(B) node[right] {$\mathbf{c}_{1,0}$}
(C) node[above] {$\mathbf{c}_{1,1}$}
(D) node[left] {$\mathbf{c}_{0,1}$}
(E) node[above] {$e$}

(AB) node[below] {$\overrightarrow{\mathbf{c}_1}$}
(BC) node[right] {$\overrightarrow{\mathbf{c}_2}$}
(CD) node[above] {$\overrightarrow{\mathbf{c}_3}$}
(AD) node[left] {$\overrightarrow{\mathbf{c}_0}$}
;

\draw[->] (A)edge(B) (B)edge(C);
\draw[->] (A)edge(D) (D)edge(C);

\end{tikzpicture}
\caption{A quadrilateral element $e$.}
\end{figure}

\begin{theorem}
With $\tilde{\mathbf{a}}$ defined in \eqref{tilde-a-def}, if the quadrilateral mesh fulfill the condition \eqref{quad-mesh-m-matr-cond}  or the mesh condition \eqref{mesh-cond}, then the stiffness matrix of the linear $Q^1$ finite element scheme \eqref{hom-var-num-quad} for solving BVP \eqref{elli-pde-1} is an $M$-matrix.
\end{theorem}

\begin{rmk}
If the diffusion coefficient matrix degenerate to a scalar, i.e. $\mathbf{a} =\alpha(\mathbf{x}) I$, a sufficient condition for \eqref{mesh-cond} is that both diagonals of the quadrilateral element bisect each angle, resulting in two non-obtuse angles for each vertex.
\end{rmk}

\begin{rmk}
By adopting some anisotropic mesh adaptation strategy where an anisotropic mesh is generated as an $M$-uniform mesh or a uniform mesh in the metric specified by the diffusion matrix $\mathbf{a}$. The method \eqref{hom-var-num-quad} for any anistropic problem possibly can be monotone on that anisotropic mesh.
\end{rmk}

If we consider rectangular meshes, for simplicity we assume 
\begin{equation*}
\mathbf{c}_{0,0} = (0,0),\quad \mathbf{c}_{1,0} = (h_1,0),\quad \mathbf{c}_{0,1} = (0,h_2),\quad \mathbf{c}_{1,1} = (h_1,h_2).
\end{equation*}
Then we have
\begin{equation*}
\tilde{\mathbf{a}} = \begin{pmatrix}
\frac{h_2}{h_1}\bar{a}^{11} & \bar{a}^{12} \\
\bar{a}^{12} & \frac{h_1}{h_2}\bar{a}^{22} 
\end{pmatrix}
\end{equation*}
and \eqref{quad-mesh-m-matr-cond} becomes
\begin{equation}\label{rect-m-cond}
|\bar{a}_e^{12}| \leq \min\{\frac{h_2}{h_1}\bar{a}_e^{11}, \frac{h_1}{h_2}\bar{a}_e^{22} \}.
\end{equation}
Recall that $\sqrt{\bar{a}_e^{11}\bar{a}_e^{22} }\ge |\bar{a}_e^{12}|$, taking 
\begin{equation}\label{anisotropic-mesh-cond}
\frac{h_1}{h_2} = \sqrt{\frac{\bar{a}_e^{11}}{\bar{a}_e^{22}}}
\end{equation} 
will guarantee \eqref{rect-m-cond}. Therefore, if the rectangular mesh is deployed with aspect ratio $\sqrt{\frac{\bar{a}_e^{11}}{\bar{a}_e^{22}}}$, then the stiffness matrix of the $Q^1$ method \eqref{hom-var-num-quad} is an $M$-matrix.

If the elliptic coefficient $\mathbf{a}$ is constant on the whole domain $\Omega$, when the rectangular mesh are fine enough, there must exist rectangular mesh with aspect ratio approximatly $\sqrt{\frac{\bar{a}_e^{11}}{\bar{a}_e^{22}}}$ such that the stiffness matrix of scheme \eqref{hom-var-num-quad} solve the BVP \eqref{elli-pde-1} is an $M$-matrix.

\begin{rmk}
Unfortunately, 
the technique in this paper cannot be easily extended to the three-dimensional case. For the three-dimensional case,  with the basis on the reference element $\hat{K}=[0,1]^3$
\begin{equation}\label{3d-basis-func}
\hat{\phi}_{i,j,k} = \hat{x}_1^i(1-\hat x_1)^{1-i}\hat{x}_2^j(1 - \hat x_2)^{1-j}\hat{x}_3^k(1 - \hat x_3)^{1-k}, \, i,j,k = 0,1,
\end{equation}
and the same derivation as in two-dimensional case, we find out
\begin{align*}
\langle \mathbf{\bar{a}}\nabla \phi_{0,0,0},\nabla \phi_{1,1,0} \rangle_{h} = -\frac{1}{16}(1+\lambda_e^3)\left[(1-\lambda_e^2)\bar{a}^{11}_e + (1-\lambda_e^1)\bar{a}^{22}_e + 2\bar{a}^{12}_e\right]+ \frac{1}{16}(1-\lambda_e^1)(1-\lambda_e^2)\bar{a}^{33}_e.
\end{align*}
For the symmetric positive-definite coefficient matrix
\begin{align*}
\mathbf{a} = \begin{pmatrix}
1 & -1+\epsilon & \epsilon\\
-1+\epsilon & 1 & -1+\epsilon \\
\epsilon & -1+\epsilon & 1
\end{pmatrix}
\end{align*}
with $\frac{1}{4}(5-\sqrt{17}) < \epsilon < \frac{1}{2}$, 
we obtain
\begin{align*}
\langle \mathbf{\bar{a}}\nabla \phi_{0,0,0},\nabla \phi_{1,1,0} \rangle_{h} = \frac{1}{16}(1+\lambda_e^3)(\lambda_e^1 + \lambda_e^2 - 2\epsilon) + \frac{1}{16}(1-\lambda_e^1)(1-\lambda_e^2) \geq \frac{1}{16}(1-2\epsilon).
\end{align*}
Then obviously the stiffness matrix fails to be an $M$-matrix. 
\end{rmk}
\section{Numerical experiment}\label{sec-test}
\subsection{Numerical experiments on uniform meshes}
In this subsection, we show tests verifying the proved order of accuracy and monotonicity of the scheme \eqref{hom-var-num-quad} on uniform rectangular meshes. 
We consider the following two-dimensional elliptic equation with Dirichlet boundary condition:  
\begin{equation} \label{test-eqn-1}
 - \nabla\cdot(\mathbf a\nabla u)+c u=f\quad \textrm{on } [0,\pi]^2
\end{equation}
where $\mathbf a=\left( {\begin{array}{cc}
   a^{11} & a^{12} \\
   a_{21} & a^{22} \\
  \end{array} } \right)$, $a^{11}=a^{12}=a_{21}=1+10x_2^2+x_1\cos{x_2}+x_2$,  $a^{22}=2+10x_2^2+x_1\cos{x_2}+x_2$ and $c = x_1^2x_2^2$, with an exact solution
  $$u(x_1, x_2)=-\sin^2{x_1}\sin{x_2}\cos{x_2}.$$

When solving this problem with our method, we take the quadrature parameter in element $e$ as $\lambda_e^1 = \lambda_e^2 =1-\frac{2|a_e^{12}|}{a_e^{11}+a_e^{22}}$.

The errors are reported in Table  \ref{elliptic-dirichlet-uniform-mesh}. We observe the desired second-order convergence in the discrete $l^2$-norm and infinity norm for the function values.  
\begin{table}[h]
\label{elliptic-dirichlet-uniform-mesh}
\centering
\caption{A two-dimensional elliptic equation with Dirichlet boundary conditions on uniform meshes. The first column is the number of elements in a finite element mesh. The second column is the number of degree of freedoms. }
\begin{tabular}{|c |c |c c|c c|}
\hline  FEM Mesh &  DoF & $l^2$ error  &  order & $l^\infty$ error & order \\
\hline
 $4\times 4$ & $3^2$ & 3.56E-1 & - & 2.70E-2 & -\\
\hline
 $8\times 8$&  $7^2$& 6.41E-2 & 2.47 & 4.89E-2 & 2.47\\
\hline
 $16\times 16$& $15^2$ & 1.49E-3 & 2.11 & 1.15E-2 & 2.08 \\
\hline
$32\times 32$ & $31^2$ & 3.65E-3 & 2.03 & 2.91E-3 & 1.99 \\ 
 \hline
 $64\times 64$ &  $63^2$ & 9.08E-4 & 2.01 & 7.25E-4 & 2.00\\
\hline
\end{tabular}
\end{table}

The monotonicity is verified by the smallest entries in $L_h^{-1}$ and $\bar{L}_h^{-1}$ which are listed in Table \ref{elliptic-unifrom-mesh}. As we can see, $L_h^{-1} \geq 0$ and $\bar{L}_h^{-1} \geq 0$ are achieved.
\begin{table}[h]
\label{elliptic-unifrom-mesh}
\centering
\caption{Minimum of entries in $\bar L_h^{-1}$ and $L_h^{-1}$  for elliptic equation \eqref{test-eqn-1} with smooth coefficients  on uniform meshes.}
\begin{tabular}{|c|c|c|}
\hline
FEM Mesh & $\bar L_h^{-1}$ & $L_h^{-1}$ \\ \hline
$4\times 4$ & $0$ & 6.38E-06  \\ \hline
$8\times 8$ & $0$ & 4.26E-10  \\ \hline
$16\times 16$ & $0$ & 2.40E-14 \\ \hline
$32\times 32$ & $0$ & 1.42E-18 \\ \hline
$64\times 64$ & $0$ & 9.24E-23 \\ \hline
\end{tabular}
\end{table}

Then we consider a more anisotropic case in the form of \eqref{test-eqn-1}
with {anisotropic-coef} 
\begin{equation}\label{anisotropic-coef}
a^{11}=1, \quad a^{12}=a_{21}=9.99, \quad a^{22}=100, \quad c = x_1^2x_2^2 
\end{equation}
and exact solution
  $$u(x_1, x_2)=-\sin^2{x_1}\sin{x_2}\cos{x_2}.$$

As stated in \eqref{anisotropic-mesh-cond}, we set $\frac{h_1}{h_2} = \sqrt{\frac{a^{11}}{a^{22}}} = 10$, then we examine the accuracy and monotonicty of the method.  When solving this problem with our method, we take the quadrature parameter in element $e$ as $\lambda_e^1 = \lambda_e^2 =1-\frac{2|\tilde a_e^{12}|}{\tilde a_e^{11}+\tilde a_e^{22}}$.

The errors are reported in Table  \ref{anisotropic-elliptic-dirichlet-uniform-mesh}. We observe the desired second-order convergence in the discrete $l^2$-norm and infinity norm for the function values.  

 \begin{table}[h]
\label{anisotropic-elliptic-dirichlet-uniform-mesh}
\centering
\caption{A two-dimensional elliptic equation with anisotropic coefficients \eqref{anisotropic-coef} and Dirichlet boundary conditions on anisotropic meshes.}
\begin{tabular}{|c |c |c c|c c|}
\hline  FEM Mesh &  DoF & $l^2$ error  &  order & $l^\infty$ error & order \\
\hline
 $40\times 4$ & $39 \times 3$ & 1.58E-1 & - & 1.20E-1 & -\\
\hline
 $80\times 8$&  $79 \times 7$& 3.59E-2 & 2.14 & 2.72E-2 & 2.14\\
\hline
 $160\times 16$& $159\times 15$ & 8.76E-3 & 2.03 & 6.65E-3 & 2.03 \\
\hline
$320\times 32$ & $319\times 31$ & 2.18E-3 & 2.01 & 1.65E-3 & 2.01 \\ 
 \hline
 $640\times 64$ &  $639\times 63$ & 5.44E-4 & 2.00 & 4.13E-4 & 2.00\\
\hline
\end{tabular}
\end{table}

The monotonicity is verified by the smallest entries in $L_h^{-1}$ and $\bar{L}_h^{-1}$ which are listed in Table \ref{anisotropic-elliptic-unifrom-mesh}. 
\begin{table}[h]
\label{anisotropic-elliptic-unifrom-mesh}
\centering
\caption{Minimum of entries in $\bar L_h^{-1}$ and $L_h^{-1}$  for elliptic equation \eqref{test-eqn-1} with anisotropic coefficients \eqref{anisotropic-coef}  on anisotropic meshes.}
\begin{tabular}{|c|c|c|}
\hline
FEM Mesh & $\bar L_h^{-1}$ & $L_h^{-1}$ \\ \hline
$40\times 4$ & $0$ & $0$  \\ \hline
$80\times 8$ & $0$ & $0$  \\ \hline
$160\times 16$ & $0$ & $0$ \\ \hline
$320\times 32$ & $0$ & $0$ \\ \hline
$640\times 64$ & $0$ & $0$ \\ \hline
\end{tabular}
\end{table}

\subsection{Numerical experiments on quadrilateral meshes}
In this subsection, we show tests verifying the proved order of accuracy and monotonicity of the scheme \eqref{hom-var-num-quad} on general quadrilateral meshes. 
We consider the following two-dimensional Poisson equation with Dirichlet boundary condition:  
\begin{equation} \label{test-eqn-2}
 - \nabla\cdot(a\nabla u)+c u=f\quad \textrm{on } [0,\pi]^2
\end{equation}
where $a=1+10x_2^2+x_1\cos{x_2}+x_2$ and $c = x_1^2x_2^2$, with an exact solution
  $$u(x_1, x_2)=-\sin^2{x_1}\sin{x_2}\cos{x_2}.$$

The domain $[0,\pi]^2$ is partitioned into  $N_y \times N_x$ elements, where the elements are forced to adapt to an inner edge. The angle between the inner edge and the $x$-axis is $\arctan(\frac{6\sqrt{3}}{5})$ as depicted in Figure \ref{quad-mesh}, where $N_y = N_x = 16$.  When solving this problem with our method, we take the quadrature parameter in element $e$ as $\lambda_e^1 = \lambda_e^2 =1-\frac{2|\tilde a_e^{12}|}{\tilde a_e^{11}+\tilde a_e^{22}}$.
\begin{figure}[ht]\label{quad-mesh}
\centering
\includegraphics[scale=0.6]{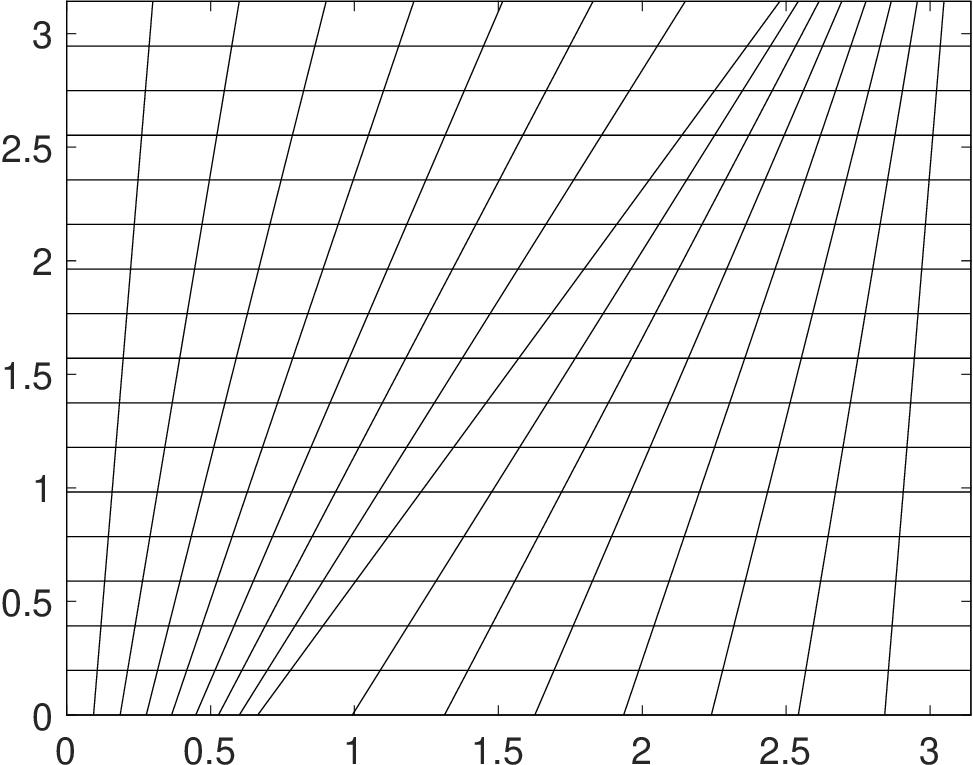}
\caption{Quadrilateral mesh.}
\end{figure}

The errors are reported in Table  \ref{elliptic-dirichlet-quad-mesh}. We observe the desired second-order convergence in the discrete $l^2$-norm and infinity norm for the function values.  

 \begin{table}[h]
\label{elliptic-dirichlet-quad-mesh}
\centering
\caption{A two-dimensional Poisson equation with Dirichlet boundary conditions on quadrilateral meshes}
\begin{tabular}{|c |c |c c|c c|}
\hline  FEM Mesh &  DoF & $l^2$ error  &  order & $l^\infty$ error & order \\
\hline
 $4\times 4$ & $3^2$ & 1.24E-1 & - & 8.70E-2 & -\\
\hline
 $8\times 8$&  $7^2$& 3.19E-2 & 1.96 & 2.84E-2 & 1.61\\
\hline
 $16\times 16$& $15^2$ & 7.82E-3 & 2.03 & 6.93E-3 & 2.04 \\
\hline
$32\times 32$ & $31^2$ & 1.94E-3 & 2.01 & 1.76E-3 & 1.97 \\ 
 \hline
 $64\times 64$ &  $63^2$ & 4.85E-4 & 2.00 & 4.41E-4 & 2.00\\
\hline
\end{tabular}
\end{table}

For the quadrilateral meshes in Figure \ref{quad-mesh}, we can verify that \eqref{quad-mesh-m-matr-cond} are satisfied on each elements numerically. Then we verify the monotonicity through the smallest entries in $L_h^{-1}$ and $\bar{L}_h^{-1}$ which are listed in Table \ref{elliptic-quad-mesh}. As we can see, $L_h^{-1} \geq 0$ and $\bar{L}_h^{-1} \geq 0$ are achieved.
\begin{table}[h]
\label{elliptic-quad-mesh}
\centering
\caption{Minimum of entries in $\bar L_h^{-1}$ and $L_h^{-1}$  for elliptic equation \eqref{test-eqn-1} on quadrilateral meshes.}
\begin{tabular}{|c|c|c|}
\hline
FEM Mesh & $\bar L_h^{-1}$ & $L_h^{-1}$ \\ \hline
$4\times 4$ & $0$ & 2.31E-4  \\ \hline
$8\times 8$ & $0$ & 1.47E-5  \\ \hline
$16\times 16$ & $0$ & 8.06E-7 \\ \hline
$32\times 32$ & $0$ & 4.64E-8 \\ \hline
$64\times 64$ & $0$ & 2.77E-9 \\ \hline
\end{tabular}
\end{table}

\section{Conclusion}
We constructed a linear monotone $Q^1$ finite element method for anistropic diffusion problem \eqref{elli-pde-1}. On uniform meshes, when the diffusion matrix is diagonally dominant, the $M$-matrix property is guaranteed thus monotonicity is achieved. When this $Q^1$ finite element method is deployed on a general quadrilateral mesh, we obtain a local mesh constraint.

\bibliographystyle{siamplain}

\bibliography{references.bib}

\end{document}